\documentclass{article}
\usepackage{cite}
\usepackage{amsthm}
\usepackage{amsfonts}
\usepackage{booktabs}
\usepackage{url}
\usepackage{multirow}
\usepackage{subcaption} 
\usepackage{threeparttable} 
\usepackage{mathtools}

\newcount\Comments  
\usepackage{xcolor}
\definecolor{darkgreen}{rgb}{0,0.5,0}

\newcommand{\kibitz}[2]{\ifnum\Comments=0\textcolor{#1}{#2}\fi}

\newcommand{\edit}[1]{\kibitz{black} {#1}} 
\newcommand{\newedit}[1]{\kibitz{black} {#1}} 

\newcommand{\R}{\mathbb{R}}
\newtheorem{definition}{Definition}
\newtheorem{theorem}{Theorem}
\newtheorem{remark}{Remark}

\newtheorem{proposition}{Proposition}
\newtheorem{corollary}{Corollary}

\title{Identifiability of car-following dynamics}
\date{}

\author{Yanbing Wang
\thanks{Department of Civil and Environmental Engineering and the Institute for Software Integrated Systems, Vanderbilt University, Nashville, TN 37240} 
\and Maria Laura Delle Monache
\thanks{Univ. Grenoble Alpes, Inria, CNRS, Grenoble INP, GIPSA-Lab, 38000 Grenoble, France and Department of Civil and Environmental Engineering, University of California, Berkeley. Present address: Department of Civil and Environmental Engineering, University of California, Berkeley, 760 Davis, MC 1710, Berkeley, CA, 94720-1710, USA}
\and Daniel B. Work\footnotemark[1]}

\begin{document}

\maketitle
\begin{abstract}
    The advancement of in-vehicle sensors provides abundant datasets to estimate parameters of car-following models that describe driver behaviors. The question of parameter identifiability of such models (i.e., whether it is possible to infer its unknown parameters from the experimental data) is a central system analysis question, and yet still remains open. This article presents both structural and practical parameter identifiability analysis on four common car-following models: \textit{i}) the constant-time headway relative-velocity (CTH-RV) model, \textit{ii}) the optimal velocity model (OV), \textit{iii}) the follow-the-leader model (FTL) and \textit{iv}) the intelligent driver model (IDM). 
    The structural identifiability analysis is carried out using a differential geometry approach, which confirms that, in theory, all of the tested car-following systems are structurally locally identifiable, i.e., the parameters can be uniquely inferred under almost all initial condition and admissible inputs by observing the space gap alone. In a practical setting, we propose an optimization-based numerical direct test to determine parameter identifiability given a specific experimental setup (the specific initial conditions and input are known). The direct test conclusively finds distinct parameters under which the CTH-RV and FTL are not identifiable under the given initial condition and input trajectory.

   \end{abstract}

\section{Introduction}
\paragraph{Motivation.} 
With the advancement of sensor technologies, abundant traffic data is readily available to study traffic patterns and individual driving behaviors. Examples of such datasets include overhead camera data~\cite{NGSIM2006,highDdataset} and floating-car data~\cite{PunzoModelCalibration2005,Treiber2013,acc_data,nuscenes2019}. The collection of such datasets support tasks such as the parameterization of microscopic models which describe individual vehicle car-following behaviors.

Studies of car-following models and calibration of such models have mostly been focused on data-fitting quality. Model calibration is usually posed as an optimization problem such that the best fit parameters are found by minimizing the error between the model prediction and corresponding measurement data~\cite{kesting2010enhanced,PunzoModelCalibration2005}, or through probabilistic approaches to find the most likely parameter candidate~\cite{wang2020estimating, abodo2019strengthening}. Although the approaches report good accuracy of the estimated parameters, they lack a theoretical guarantee that a unique parameter set can be recovered, which is provided by an identifiability analysis of the evolution-observation system~\cite{BELLMAN1970329,ljung1999system}. In other words, using data-fitting quality as the metric alone for model calibration does not guarantee the uniqueness of the best-fit parameters, nor can it tell how robust the current estimation is under a different experimental dataset or numerical setup~\cite{ljung1999system}. Such ambiguity might result in a different parametrization when repeating the experiment, leading to an inaccurate characterization of car-following behaviors, such as analytically determining string stability~\cite{monteil2018mathcal}.

\edit{
Besides that different driver behaviors on a microscopic scale can greatly change traffic features on a macroscopic scale~\cite{LAVAL2014228,LI2017headway,PunzoModelCalibration2005}, knowledge of microscopic model parameters is key to understand individual driving behavior of the adaptive cruise control  and the interplay between automated vehicles and human drivers~\cite{Bando1995dynamical,Gazis1959,gunter2019modelcomparison,milanes2014modeling,milanes2014cooperative,wang2020estimating,HE2015185,kesting2010enhanced}. 
Therefore, the identifiability of car-following models is an important question to advance calibration of these models.}

\paragraph{Problem statement.}
The parameter identifiability problem refers to the ability to uniquely recover the unknown model parameters from the observation data. Mathematically, consider a general form of a dynamical system:
\begin{equation}
    \label{eq:dyn_sys_general}
    \left\lbrace
    \begin{array}{l}
         \dot{x}(t)= f(x(t), u(t),\theta)  \\
         y(t)= g(x(t), u(t), \theta)\\
         x(0) = x_0,
    \end{array}
    \right.
\end{equation}
where $x(t) \in \R^{n_x}$ is the state variable vector, $u(t)$ is the scalar input, and $\theta \in \R^{n_{\theta}}$ is the parameter vector; $f,\ g$ are analytic vector functions describing the evolution of the state and the measurement model, and $x_0$ is the initial condition. The generic question for parameter identifiability is: \textit{Given the system \eqref{eq:dyn_sys_general} and a known input $u(t)$, can we uniquely determine the model parameters $\theta$ from the output $y(t)$ for a given initial condition?}
In the context of car-following models, the system of equations can be written as:
\begin{equation}
    \label{eq:dyn_sys_CF}
    \left\lbrace
    \begin{aligned}
    \dot{x}(t) &= 
    \begin{bmatrix}
    \dot{s}(t)\\
    \dot{v}(t)
    \end{bmatrix}
    =
    \begin{bmatrix}
    u(t)-v(t)\\
    f_{\text{CF}}(x(t), u(t), \theta)  
    \end{bmatrix}\\
     y(t) &= x(t)\\
     x(0) &= \begin{bmatrix}
    s_0\\
     v_0
    \end{bmatrix}.
     \end{aligned}
      \right.
\end{equation}
 The state vector $x(t)$ in \eqref{eq:dyn_sys_CF} is given by the space gap $s(t)$  and the velocity $v(t)$ of the ego vehicle.  The space gap is taken as the distance between the lead vehicle and the ego (follower) vehicle while the speed is that of the ego vehicle. The input $u(t)$ is the velocity profile of a lead vehicle and $f_{\text{CF}}$ is usually a 2$^{\text{nd}}$ order ordinary differential equation (ODE) that describes the car-following dynamics.

With in-vehicle sensor data such as radar data being available, it is now possible to directly measure $s(t)$ and $v(t)$, as well as the velocity profile of the lead vehicle  $u(t)$. 
For \eqref{eq:dyn_sys_CF}, the parameter identifiability question becomes: \textit{Given the output time-series $y(t)$ (the velocity and space gap measurements) and the input $u(t)$, can we uniquely determine the model parameters $\theta$ in $f_{\text{CF}}$?}

One type of identifiability is \textit{structural identifiability} (or theoretical identifiability), which investigates the input-output configuration and relies on algebraic calculations to determine if the parameter set is unique from the observation~\cite{walter1997identification,LJUNG1994265}. It is completely defined by the model structure and does not consider data quality or model errors. A \textit{structural identifiability} analysis can be performed prior to the estimation of the unknowns in order to detect potential model structural issues and separate them from numerical problems, which should be dealt with differently~\cite{villaverde2019full}. However, structurally identifiable systems do not preclude the existence of some specific initial conditions under which the system becomes unidentifiable~\cite{saccomani2003parameter,Villaverde2017initial}. For this limitation, \textit{practical identifiability} can be complimentary to provide insights on the dynamical models.

\textit{Practical identifiability} is usually based on numerical methods to determine, for a given experimental setup, whether the parameters are unique. This is sometimes modified to also consider the presence of moderate measurement errors. In practice, a realistic way to formulate the identifiability problem when performing an experiment becomes: \textit{for a given initial condition and an input trajectory that are known, whether there exists distinct valid parameters that produce the same outputs?} \edit{Notice that ensuring uniqueness of the parameters (e.g. there are no distinct parameters that produce the same output) does not imply well-posedness of the inverse problem. When solving the inverse problem, one needs to ensure additional properties, e.g. continuity w.r.t. the initial datum, to guarantee the recovery of the parameters and reconstruction the input.}

This paper investigates, from both theoretical and practical angles, parameter identifiability of a microscopic car-following system instantiated as one of four different car-following models, i.e., the constant-time headway relative-velocity (CTH-RV) model~\cite{Bando1995dynamical,milanes2014modeling,gunter2019modelcomparison}, the optimal velocity (OV) model~\cite{Bando1995dynamical}, the follow-the-leader (FTL) model~\cite{garavello} and the intelligent driver model (IDM)\cite{treiber2000congested}. Only CTH-RV is a linear model while other three are nonlinear. We address the following questions:
\begin{enumerate}
    \item Are the car-following systems of the form~\eqref{eq:dyn_sys_CF} structurally identifiable?
    \item Are they identifiable in practice given a specific experimental setup (i.e., given a known initial condition and a known input)?
    \item Are they identifiable in practice given a specific experimental setup when moderate measurement errors are allowed?
\end{enumerate}
We address the question of structural identifiability using a \textit{differential geometry approach}, and practical identifiability using a \textit{numerical direct test}. The main findings to the above three questions are:
\begin{enumerate}
    \item The investigated car-following models are all structurally locally identifiable, i.e., almost all points in the parameter space of each of the car following models can be uniquely identified for all admissible inputs given almost \newedit{all} initial conditions. Special initial conditions and parameter sets that lose structural identifiability are also discovered.
    \item Given a specific experimental setup (the specific initial condition and input that are known), the CTH-RV and FTL are not identifiable, i.e., there exist distinct parameters that produce the same output.
    \item Given a specific experimental setup when measurement errors are present, all models are not practically identifiable, i.e., there exist distinct parameters that produce the same output within a small error bound. 
\end{enumerate}

\paragraph{Contributions.}
 Throughout the literature of microscopic traffic modeling we found that the focus has been on model calibration techniques instead of parameter identifiability. In this light, the main contributions of this article are the following. (\textit{i}) We provide structural identifiability analysis of four common car-following models using a differential geometry framework. This allows us to detect if there are structural issues of the dynamic model and to distinguish them from other possible causes of calibration failures, such as the choice of optimization algorithms. (\textit{ii}) We subsequently provide a numerical \textit{direct test}, which is a constrained optimization problem, to check identifiability for a specific experimental design (for a given and known initial condition and input; (\textit{iii}) finally we investigate the practical identifiability for a specific experimental design when moderate measurement errors are present,  also with a \textit{direct test}.  This study therefore provides both analytical and practical insights on parameter identification.

The organization of this paper is as follows. In Section~\ref{sec:related} we summarize the existing microscopic traffic model calibration work, as well as methods derived in other research fields to tackle parameter identifiability problems. In Section~\ref{sec:preliminaries} we present the four microscopic car-following models that we analyze in this work. We introduce the methodological tools used, including the \textit{differential geometry} approach and the \textit{numerical direct test} approach, in Section~\ref{sec:method}. Section~\ref{sec:structural_id_results} provides the structural identifiability results of the selected car-following models, and Section~\ref{sec:practical_id_results} presents the practical identifiability results.

\section{Related work}
\label{sec:related}
The inverse problem of learning microscopic car following model parameters through experimental data (e.g., microscopic-level trajectory data and macroscopic level aggregated counts from fixed sensors) is part of the process to calibrate complex microscopic traffic simulation software.  A recent review of these calibration techniques appears in~\cite{LI2020225}, and is commonly categorized  into four types~\cite{panwai2005comparative,Punzo2012CanRO,punzo2015do,LI2017headway}. Briefly put, \textit{type I} calibration views the calibration problem as a likelihood estimation problem where the distribution of the parameter likelihood is calculated for the future time step based on historical driving data (e.g., \cite{panwai2005comparative, Treiber2013, hoogendoorn2010generic,TREIBER2013922,punzo2015do,PunzoModelCalibration2005, wang2020estimating}). \textit{Type II} calibration directly uses a global search to find the best-fit parameters for which the simulated complete trajectory most closely represent the observed trajectory (e.g., \cite{ma2002genetic,wang2010usingTraj, ciuffo2014nofree,Papathanasopoulou2015towards,gunter2019modelcomparison}). \textit{Type III} calibration considers the long-range interactions amongst vehicles within a platoon (e.g., \cite{LAVAL2014228, Kurtc2015calibratingLocal,HE2015185}), and \textit{type IV} calibration relies on mesoscopic or macroscopic traffic flow patterns such as the headway distributions (e.g., \cite{JIN2009318,LI2017headway}). This paper considers parameter identifiability of microscopic models given microscopic data (e.g., velocity and spacing), which are closely related to \textit{Type I} and \textit{II}. More background on these \textit{Type I} and \textit{II} calibration methods can be found in book chapters such as~\cite{Treiber2013,Chen2015stochasticBook}, review articles~\cite{kesting2008calibrating,Hollander2008principles,hoogendoorn2010generic,LI2016global,LI2020225} and the references therein.



In addition to the question of determining the best fit parameters given experimental data, several studies have also considered the confidence level and the sensitivity of parameter estimation. For example, Punzo et al.~\cite{punzo2015do} considered a variance-based global sensitivity analysis to produce the importance ranking of the IDM parameters, and to consequently reduce the parameters to be further calibrated. Monteil \& Bouroche~\cite{Monteil_identifiability} considered a systematic statistical approach to first use a global sensitivity analysis to reduce the parameter space, then log likelihood estimation for the insensitive parameters and finally likelihood-ratio for interval estimation; Treiber \& Kesting~\cite{TREIBER2013922} investigated on the data sampling interval, completeness and parameter orthogonality and their effects on parameter calibration of the IDM, instead of solely on data-fitting quality. The sensitivity approaches considered in the mentioned works are closely related to identifiability~\cite{miao2011,Monteil_identifiability}. However, a formal analysis of identifiability on microscopic traffic models (from a theoretical perspective) has not yet been explicitly studied in the previous work.

Although a systematic study on the parameter identifiability for microscopic traffic models still remains unaddressed in transportation research, it has been developed in and extensively applied to other research fields such as waste water treatment process~\cite{oliver2001waste}, robot dynamics~\cite{Khosla2007robot}, and biological processes~\cite{TUNCER20181}, where the system evolution equations are also expressed as ODEs. We summarize the related approaches for tackling structural and practical identifiability respectively.

Structural identifiability of dynamical systems is closely related to the algebraic formulation of the dynamic equations, and provides theoretical possibility for uniquely inferring the system unknowns \textit{a priori} before collecting experimental data~\cite{glad1997solvability,LJUNG1994265}. It can be analyzed using similarity transformation~\cite{WALTER19811,CHAPPELL1992241,Hermann,VAJDA1989217} (applicable to autonomous systems), Laplace transform~\cite{BELLMAN1970329}, power series expansion~\cite{pohjanpalo,Grewal}, implicit functions~\cite{xia2003}, differential algebra~\cite{Ritt1950, GerdtImproved, LJUNG1994265} and differential geometry~\cite{villaverde2016structural} (applicable to systems with external inputs). Multiple software packages have also been developed for direct implementation of some of the techniques, for example, DAISY~\cite{Bellu2007DAISY}, COMBOS~\cite{Meshkat2014COMBOS}, SIAN~\cite{hong2019global} and STRIKE-GOLDD~\cite{villaverde2019full,Villaverde2019input}. All the mentioned tools are suitable for analyzing systems of which dynamics are written as rational functions. STRIKE-GOLDD, in particular, is capable for analyzing nonlinear, non-rational system dynamics as well, making it suitable for analyzing a wider class of car-following models.


In practice when an experiment is fully defined (the specific values of the initial condition and the input trajectory are known), numerical methods are often useful to assess identifiability. The numerical methods can be used to explore the profile likelihood~\cite{Raue2009profile,kreutz2018easy}, which allows to derive likelihood-based confidence intervals for each parameter and recovers the functional relations between parameters due to non-identifiability. The sensitivity matrix~\cite{STIGTER2015118} combines numerical calculations with a tractable symbolic computation to investigate local structural identifiability.

Numerical methods can also help to assess practical identifiability, which relaxes the noise-free assumption of $f$ and $g$ in the structural analysis. Methods such as Monte Carlo simulations~\cite{metropolis} help to check the relative error of estimated parameters under noisy output measurements. Numerically evaluating the profile likelihood can help to distinguish practically unidentifiable parameters due to measurement error from structurally unidentifiable ones by the shape of the likelihood profile~\cite{maiwald2016profile}..

\section{Car-following models}\label{sec:preliminaries}  
In this section, we briefly present four commonly used car-following models that describe the acceleration  dynamics in the form of an ODE. 


\paragraph{Constant-time headway relative-velocity (CTH-RV) model}
The CTH-RV model has been predominately used to describe the car-following behaviors of ACC vehicles~\cite{milanes2014modeling,milanes2014cooperative,gunter2019modelcomparison,wang2020estimating,Bareket2003methodology,Liang2010vehicle}. It is a simple model with linear dynamics with respect to the space gap and relative velocity $\Delta v(t):= u(t)-v(t)$:
\begin{equation}
\label{eq:CTH-RV}
\dot{v}(t) = k_1(s(t)- \tau  v(t))+k_2(\Delta v(t)),
\end{equation}
where $k_1$ and $k_2$ are non-negative gains and $\tau $ is the constant time-headway. The three (time-invariant) parameters constitute the parameter vector of this model, i.e., $\theta = [k_1, k_2, \tau]$, whose identifiability is to be determined.

The equilibrium initial condition $x_0^*$ for the car-following system~\eqref{eq:dyn_sys_CF} under the acceleration dynamic~\eqref{eq:CTH-RV} is: 
\begin{equation}
    \label{eq:CTH-RV-equilibrium}
    x_0^*=[\tau u_0, u_0]^T,
\end{equation}
where $u(0)=u_0$ is the initial velocity of the leader. It is easy to see that this initial condition results in $\dot{s}(0)=u_0-v_0=0$ and $\dot{v}(0)=f_{CF}(x_0^*,u_0, \theta)=0$. 

\paragraph{Optimal velocity (OV) model}
We also consider the car-following model proposed by Bando et al.~\cite{Bando1994StructureSO,Bando1995dynamical}, where an optimal velocity term $V(s(t))$ is introduced to describe the desired spacing-speed relationship at equilibrium:
\begin{equation}
    V(s(t)) = a\left( \tanh\left(\dfrac{s(t)-h_m}{b}\right)+\tanh\left(\dfrac{h_m}{b}\right)\right),
\end{equation}
where the parameters $a,h_m,$ and $b$ determine the shape of the optimal velocity function, which increases monotonically as a function of $s(t)$, and asymptotically plateaus at a maximum speed as $s\rightarrow \infty$.
Consequently the vehicle accelerates and decelerates to achieve the optimal velocity:
\begin{equation}
    \label{eq:bando}
    \dot{v}(t) =
\alpha\left(V(s(t))-v(t)\right),
\end{equation}
where the parameter $\alpha$ determines the sensitivity of the stimulus, which is the difference between the desired velocity $V(s(t))$ and the actual velocity $v(t)$.  The parameter vector of this model is $\theta = [\alpha, a, h_m, b]$.  The corresponding equilibrium initial condition for~\eqref{eq:dyn_sys_CF} under~\eqref{eq:bando} is: 
\begin{equation}
    \label{eq:OV-equilibrium}
    x_0^*=\left[h_m- b*\mathrm{tanh}^{-1}\left(\mathrm{tanh}\left(\frac{h_m}{b}\right)-\frac{u_0}{a}\right), u_0\right]^T.
\end{equation}




\paragraph{Follow-the-Leader (FTL) model}


The third model is one of the simplest follow-the-leader variations of the Gazis-Herman-Rothery (GHR) car-following model, which originated from research conducted by General Motors in the 1950s~\cite{Gazis1959,GHR,garavello}. We consider the following form:
\begin{equation}
    \label{eq:FTL}
    \dot{v}(t) =
     \dfrac{C\Delta v(t)}{s(t)^\gamma},
\end{equation}
where the parameter $C$ and $\gamma$ are constants describing the sensitivity of $\Delta v(t)$ and $s(t)$, respectively, or the acceleration. The parameter vector of interest is $\theta = [C, \gamma]$. The corresponding equilibrium initial condition $x_0^*$ for~\eqref{eq:dyn_sys_CF} under~\eqref{eq:FTL} is:
\begin{equation}
    \label{eq:FTL-equilibrium}
    x_0^*=[s_0, u_0]^T.
\end{equation}

\paragraph{Intelligent driver model (IDM)}
The intelligent driver model was proposed in~\cite{treiber2000congested} to model a realistic driver behavior, such as asymmetric accelerations and decelerations. It is of the form:
\begin{equation}
\label{eq:Enhanced_IDM}
\dot{v}(t) = a\left[1-\left(\dfrac{v(t)}{v_f}\right)^{4}-\left(\dfrac{s^*(v(t),\Delta v(t))}{s(t)}\right)^2\right]
\end{equation}
where the desired space gap $s^*$ is defined as:
\begin{equation}
    s^*(v(t),\Delta v(t)) = s_j + v(t) T + \dfrac{v(t)\Delta v}{2\sqrt{ab}}.
\end{equation}
The parameters of the model $\theta = [s_j, v_f, T, a, b]$ are the freeflow speed $v_f$, the desired time gap $T$, the jam distance $s_j$, the maximum acceleration $a$ and the desired deceleration $b$. The corresponding equilibrium initial condition $x_0^*$ is:
\begin{equation}
    \label{eq:IDM-equilibrium}
    x_0^*=\left[\frac{s_j+u_0*T}{\sqrt{1-(u_0/v_f)^4}}, u_0\right]^T.
\end{equation}


\section{Methodology}
\label{sec:method}
In this section we highlight the methods for identifiability analysis. The structural local identifiability is carried out through a \textit{differential geometry} method, and the practical identifiability considering a full experimental setup and measurement error is examined using a \textit{numerical direct test} method.

\subsection{Differential geometry framework for structural identifiability}
We first introduce the concepts of structurally global and structurally local identifiability via precise definitions. \newedit{Note that the following definitions are modified from references such as~\cite{villaverde2019full, Grewal,LJUNG1994265,WALTER1982472}, in order to suit the specific format of our analysis.}\\ 


\begin{definition}\label{def:struct_glob_ident}
\edit{Let $\theta \in \rm I\!R^{n_p}$ denote the generic parameter vector, \newedit{$\mathcal{X}_0$ a set of generic initial conditions} and $\mathcal{U}$ a set of admissible inputs. Let $y(t,\theta, x_0,u )$ be the output function from the state-space model~\eqref{eq:dyn_sys_general}. If for all $t>0$, }
\newedit{
\begin{equation}
    \label{eq:struct_glob_2}
    y(t,\theta, x_0, u) \equiv y (t,\theta^*, x_0, u) \Rightarrow \theta=\theta^*
\end{equation}
}for almost \newedit{all} $\theta\in \rm I\!R^{n_p}$, \newedit{almost all $x_0\in \mathcal{X}_0$} and every input $u \in \mathcal{U}$ then the model is said to be structurally globally identifiable.
\end{definition}

\begin{definition}
\label{def:struct_loc_ident}
\edit{A dynamical system given by (1) is structurally locally identifiable (s.l.i.) if for almost all $\theta\in {\rm I\!R}^{n_{p}}$ there exists a neighborhood $\mathcal{N}(\theta)$ such that, for all $\theta_1,\theta_2\in \mathcal{N}(\theta)$, the implication~\eqref{eq:struct_glob_2} holds for all $t>0$.}
\end{definition}

\begin{remark}
\newedit{Notice that Definitions 1. and 2. consider the generic parameter vector and initial conditions, and thus may not be applicable to parameters and initial conditions that fall into a measure zero set.}
\end{remark}
\begin{remark}
\edit{We consider admissible inputs are polynomial inputs of degree $n$ (with $n$ non-zero time derivative) that enables structurally identifiable systems. In particular, no admissible input exists for structurally unidentifiable systems.}
\end{remark}


The differential geometry framework~\cite{villaverde2018observability} considers identifiability as an augmented observability property for a general nonlinear system of ODEs such as~\eqref{eq:dyn_sys_general}, and it can be evaluated in the same manner.
The main idea \edit{of identifiability} consists in considering the parameters $\theta$ as additional states with zero dynamics. Hence, the parameter-augmented state and the dynamics become \newedit{$\tilde{x} = [x, \theta]\in \R^{n_{\tilde{x}}}$, with ${n_{\tilde{x}} = n_x+n_{\theta}}$} and
\begin{equation}
\label{eq:aug_state}
    \tilde{x}(t) = \begin{bmatrix}
    x(t)\\
    \theta
    \end{bmatrix}
    \qquad \text{and} \qquad
    \left\lbrace
    \begin{array}{l}
         \widetilde{\dot{x}}(t) = f(\tilde{x}(t),u(t))  \\
         y(t) = g(\tilde{x}(t),u(t))\\
         \tilde{x}(0) = \tilde{x}_0=[x_0, \theta]^T.
    \end{array}
    \right.
\end{equation}

In the general case of a \edit{nonlinear system} with time-varying input, to identify the parameters we need to take into account the changes to the output due to the changing input in the augmented state. Thus we use the extended Lie derivatives. The extended Lie derivatives can, then, be used to build the rows of the observability-identifiability matrix $\mathcal{O}_I$~\cite{KARLSSON2012941,Villaverde2019input}\edit{, whose rank is to be evaluated for the identifiability analysis}. The Lie derivative of $g(\tilde{x}(t), u(t))$ in the direction of $f(\tilde{x}(t), u(t)) $ is given by:

\begin{equation}\label{eq:lie_derivative}
    L_f g (\tilde{x},u) = \dfrac{\partial g (\tilde{x}, u)}{\partial \tilde{x}}f(\tilde{x},u) + \sum\limits_{j=0}^{j=\infty} \dfrac{\partial g (\tilde{x}, u)}{\partial u^{(j)}}u^{(j+1)}
\end{equation}
where $u^{(j)}$ is the $j^{th}$ time derivative of the input $u$. The Lie derivatives of higher order are:
\begin{equation}\label{eq:lie_derivative_high_order}
    L_f^i g (\tilde{x},u) = \dfrac{\partial L_f^{i-1} g (\tilde{x}, u)}{\partial \tilde{x}}f(\tilde{x},u) + \sum\limits_{j=0}^{j=\infty} \dfrac{\partial L_f ^{i-1} g (\tilde{x}, u)}{\partial u^{(j)}}u^{(j+1)}.
\end{equation}
The observability-identifiability matrix $\mathcal{O}_I (\tilde{x},u)$ for a general nonlinear system~\eqref{eq:aug_state} becomes 
\begin{equation}\label{eq:observability_matrix}
    \mathcal{O}_I (\tilde{x},u) = \begin{pmatrix}
    \frac{\partial}{\partial \tilde{x}} g (\tilde{x}, u)\\
    \frac{\partial}{\partial \tilde{x}} (L_f g (\tilde{x}, u))\\
    \frac{\partial}{\partial \tilde{x}} (L_f^2 g (\tilde{x}, u))\\
    \vdots\\
    \frac{\partial}{\partial \tilde{x}} (L_f^{n_{\tilde{x}}-1} g (\tilde{x}, u))\\
    \end{pmatrix}
\end{equation}

Finally, a nonlinear observability-identifiability condition can be used to infer the structural local identifiability of~\eqref{eq:aug_state}:
\begin{theorem}[Nonlinear Observability – Identifiability Condition (OIC)]\label{theorem:oic}
If a model \eqref{eq:aug_state} satisfies $\mathrm{rank}(\mathcal{O}_I (\tilde{x}_0,u)) = n_x+n_{\theta}$, with $\mathcal{O}_I (\tilde{x}_0,u)$ given by \eqref{eq:observability_matrix} and $\tilde{x}_0$ being a (possibly generic) point in the augmented state space, the model is locally observable and locally structurally identifiable in a neighborhood $\mathcal{N}(\tilde{x}_0)$ of $\tilde{x}_0$.
\end{theorem} 

\begin{remark}
The differential geometry approach yields results that are
valid almost everywhere, i.e., for all values of the system variables (initial conditions and inputs) except for a set of measure zero.
\end{remark} 

Throughout the rest of the paper, we focus on analyzing structural local identifiability for almost all initial conditions in the augmented initial state $\tilde{x}_0$, i.e., we compute $\mathcal{O}_I (\tilde{x}_0,u)$ symbolically with $\tilde{x}_0=[x_0, \theta]^T$. We also explore the special cases for $\tilde{x}_0$ (the initial conditions and parameters that belong to a set of measure zero) such that structural identifiability is lost. More details of the setup and analysis are presented in Section~\ref{sec:structural_id_results}.

\subsection{Numerical direct test}


The \textit{direct test} is a conceptually straightforward way to test identifiability considering a fully-defined system (i.e., the given initial condition and input trajectory are known). The direct test offers analysis that differs from structural identifiability in that it looks for the worst-case unidentifiable parameters in the following sense. It poses the identifiability problem as for finding the maximally distinct (as specified by an objective function) parameters $\theta_1$ and $\theta_2$ such that the output of two systems (which are identical except for the parameters) differ by no more than a threshold $\epsilon$ (possibly equal to zero). It is an idea first proposed in~\cite{walter2004guaranteed} as a numerical alternative to algebraic computation.  

The generic form of the direct test reads:  

\begin{equation}
\label{eq:opt2}
\begin{aligned}
& \underset{\theta_1,\theta_2\in \Theta}{\text{maximize}}
& & d(\theta_1, \theta_2) \\
& \text{subject to}
& & e(\theta_1, \theta_2)\leq\epsilon,
\end{aligned}
\end{equation}
where the objective function $d(\theta_1, \theta_2)$ is the distance between two parameters in the parameter space $\Theta$. The  constraint $e(\theta_1, \theta_2) \leq \epsilon$ caps the difference between the output of the system under $\theta_1$ and $\theta_2$ at $\epsilon$, which is a small, user-defined threshold of measurement error. In order to evaluate $e(\theta_1,\theta_2)$, one must solve the ODE~\eqref{eq:dyn_sys_general} for the same known initial condition and the same given input trajectory under the two parameters $\theta_1$ and $\theta_2$. This problem formulation finds the most distinct (as quantified by $d$) parameters such that the fully defined system produces similar outputs (or the same outputs when $\epsilon=0$). The difference $e(\cdot,\cdot)$ is only due to $\theta$ for a fully-defined system. Therefore $e(\cdot,\cdot)$ is a function of only two parameters $\theta_1$ and $\theta_2$.


Denote the decision variables of~\eqref{eq:opt2} at optimality as $\theta_1^*,\theta_2^*$, and the objective function value at optimality as $\delta^*:=d(\theta_1^*,\theta_2^*)$.
Naturally, under the tightest constraint $\epsilon=0$, all parameters in $\Theta$ for system (1) with a specific and known initial condition and input trajectory produce a unique output if and only if $\delta^*=0$. When this occurs, we say that the system~\eqref{eq:dyn_sys_general} with the prescribed initial and input condition is identifiable; the same system is unidentifiable under the same initial and input conditions \edit{if $\delta^*>\delta$, where $\delta$ is a small positive number that is problem-specific.}

When measurement error is allowed, we consider a relaxed constraint, $\epsilon>0$. If $\delta^*$ is small, then we say the system is practically identifiable, i.e., similar outputs can be produced only by parameters that are similar. On the other hand, if $\delta^*$ is large, we found two distinct parameters that generate similar outputs, indicating that the same system is practically unidentifiable. 

Results on the change of $\delta^*$ as $\epsilon$ changes provides insights on the sensitivity of identifiability with respect to the measurement error, and will be presented in the numerical experiment (Section~\ref{sec:practical_id_results}).

\section{Structural identifiability analysis}
\label{sec:structural_id_results}

In this section, we analyze structural local identifiability of the car-following system~\eqref{eq:dyn_sys_CF} under the various car following models in Section~\ref{sec:preliminaries}. The analysis reveals that all the systems are structurally locally identifiable, i.e., it is theoretically possible to uniquely infer the unknown parameters under any admissible input and almost \newedit{all} initial conditions. We also explore the specific initial conditions (those belonging to a set of measure zero) such that each model becomes unidentifiable.

The organization of this section is the following. First we give an overview of the analysis, including the toolbox for implementing the differential geometry analysis for a given car following model, and interpret the results of the analysis. Next, we provide the structural local identifiability analysis for the CTH-RV~\eqref{eq:CTH-RV} model in detail, by showing the explicit observability-identifiability matrices $\mathcal{O}_I (\tilde{x}_0,u)$ under both a generic and specific initial conditions, and a set of admissible inputs for which the system is structurally identifiable. The structural local identifiability results of the other models (\eqref{eq:bando}, \eqref{eq:FTL} and \eqref{eq:Enhanced_IDM}) are then summarized. 

\subsection{Implementation overview}

 We deploy a direct implementation of the framework for structural identifiability through \textit{STRuctural Identifiability taKen as Extended-Generalized Observability with Lie Derivatives and Decomposition} (STRIKE-GOLDD)~\cite{villaverde2016, villaverde2019full}, an open-source MATLAB toolbox that computes the observability-identifiability matrix $\mathcal{O}_I (\tilde{x}_0,u)$ and analyzes the local structural parameter identifiability, state observability, and input reconstructability of nonlinear dynamic models of ODEs. Because only parameter identifiability is concerned in this paper as the states and the inputs can be directly measured, we use the toolbox solely for the purpose of identifiability. 

In addition, we use STRIKE-GOLDD to also determine \textit{a priori} the \edit{minimum degree $n$ of a polynomial input $u(t)$ to enable a structurally identifiable system~\cite{Villaverde2019input}. For example, if $\mathcal{O}_I$ has full rank under $u^{(n)}(t)=0$ but not $u^{(n-1)}(t)=0$, then an input $u(t)$ of degree $n$ or above is sufficient for structural identifiability.  Otherwise, $u(t)$ with degree higher than $n$ should be explored. }

For structurally unidentifiable systems, no admissible input exists. For structurally identifiable systems, the \edit{degree $n$ (i.e., number of non-zero time derivatives)} for an input to be admissible depends on the initial condition. We explore the admissible input condition for structurally identifiable car-following systems under both generic initial conditions and a special set of initial conditions where higher-order input may be required to be admissible.

\subsection{Structural local identifiability analysis for CTH-RV}
Recall CTH-RV ODE~\eqref{eq:CTH-RV}. The parameter-augmented system dynamics can be written as
\begin{equation}
    \label{eq:CTHRV_augmented}
    \left\lbrace
    \begin{array}{l}
         \widetilde{\dot{x}}_1(t) = \dot{s}(t) = u(t)-v(t) \\
         \widetilde{\dot{x}}_2(t) = \dot{v}(t) = k_1(s(t)-\tau v(t))+k_2(u(t)-v(t))\\
         \widetilde{\dot{x}}_{3-5}(t) = \dot{\theta}(t) = 0\\
         y(t) = s(t)\\
         \tilde{x}(0) = [x_0, \theta]^T.
    \end{array}
    \right.
\end{equation}
The following propositions and proofs show the structural identifiability of CTH-RV from the calculated observability-identifiability matrices.

\begin{proposition}\label{prop1}
According to Definition~\ref{def:struct_loc_ident}, the CTH-RV system~\eqref{eq:CTHRV_augmented} is structurally locally identifiable under almost all initial condition $x_0$ and an admissible input $u(t)$ \edit{with degree $n\geq0$} up to a set of measure zero.
\end{proposition}

\begin{proof}
Consider \newedit{almost all} initial state conditions $x_0 =  [s_0, v_0]^T$. We prove that input $u(t)$ \edit{with degree} $n\geq0$ is admissible for structural local identifiability, i.e., $\mathcal{O}_I (\tilde{x}_0,u)$ is full rank with constant input $u(t)=u_0$. This implies that $\mathcal{O}_I (\tilde{x}_0,u)$ for $n>0$ is also full rank~\cite{Villaverde2019input}.

Constructing the observability-identifiability matrix $\mathcal{O}_I (\tilde{x}_0,u)$ requires at least $n_{\tilde{x}}-1=4$ extended Lie Derivatives, where $n_{\tilde{x}}=5$ is the dimension of the augmented state. The four extended Lie Derivatives are:
\begin{equation}\label{eq:lie_derivatives_proof1}
    \begin{aligned}
        L_f^1 g (\tilde{x},u) &=u(t)-v(t)\\
        L_f^2 g (\tilde{x},u) &=k_2(v(t) - u(t)) - k_1(s(t) - \tau v(t))\\
        L_f^3 g (\tilde{x},u) &=k_1(v(t) - u(t)) - (k_2 + k_1\tau)(k_2(v(t) - u(t)) - k_1(s(t) - \tau v(t)))\\
        L_f^4 g (\tilde{x},u) &=-(k_1 - (k_2 + k_1\tau)^2)(k_2*(v(t) - u(t))- k_1(s(t) - \tau v(t)))\\
        &- k_1(k_2 + k_1\tau)(v(t) - u(t)).
    \end{aligned}
\end{equation}
Consequently, $\mathcal{O}_I (\tilde{x}_0,u)$ for system~\eqref{eq:CTHRV_augmented} is a 5$\times$5 matrix:
\begin{equation}
\label{eq:O_1}
\begin{aligned}
    &\mathcal{O}_I (\tilde{x}_0,u) = \\
    &\begin{bmatrix}
    1 & 0 & 0 & 0 & 0\\
    0 &-1 & 0 & 0 & 0\\
    -k_1 & k_2+k_1\tau & \tau v_0-s_0 & v_0-u_0 & k_1 v_0\\
    -k_1(k_2+k_1\tau) & (k_2+k_1\tau)^2-k_1 & o_{43} & o_{44}& o_{45}&\\
    o_{51}& o_{52}& o_{53}&o_{54}& o_{55}&
    \end{bmatrix}
    \end{aligned}
\end{equation}
with specific entries defined as follows:
\begingroup
\allowdisplaybreaks
\begin{align*}
o_{43} &= v_0-u_0+(k_2+k_1\tau)(s_0-\tau v_0)-\tau(k_2(v_0-u_0))-k_1(s_0-\tau v_0) \\
o_{44} &= k_1(s_0 - \tau v_0) - k_2(v_0 - u_0) - (k_2 + k_1\tau)(v_0 - u_0)\\
o_{45} &=- k_1 (k_2 (v_0 - u_0) - k_1 (s_0 - \tau v_0)) - k_1 v_0 (k_2 + k_1 \tau)\\
o_{51} & =k_1 (k_1 - (k_2 + k_1 \tau)^2)\\
o_{52} & =- k_1 (k_2 + k_1 \tau) - (k_2 + k_1 \tau) (k_1 - (k_2 + k_1 \tau)^2)\\
o_{53} &=(s - \tau v_0) (k_1 - (k_2 + k_1 \tau)^2) - (k_2 + k_1 \tau) (v_0 - u_0) + (2 \tau (k_2 + k_1 \tau) - 1) \\&(k_2 (v_0 - u_0) - k_1 (s_0 - \tau v_0)) - k_1 \tau (v_0 - u_0)\\
o_{54} & =(2 k_2 + 2 k_1 \tau) (k_2 (v_0 - u_0) - k_1 (s_0- \tau v_0)) - \\
&k_1 (v_0 - u_0) - (k_1 - (k_2 + k_1 \tau)^2) (v_0 - u_0)\\
o_{55} &= k_1^2 (v_0 - u_0) - k_1 v_0 (k_1 - (k_2 + k_1 \tau)^2).
\end{align*}
\endgroup
\edit{Analytically, it is difficult to prove that the matrix $\mathcal{O}_I (\tilde{x}_0,u)$ is full rank. A symbolic calculator such as MATLAB can provide a general case: the matrix in general is not rank deficient (see an example in Appendix~\ref{appendix:example_rank5}). In this case, the system is structurally locally identifiable according to Theorem~\ref{theorem:oic}.} Since constant input $u(t) = u_0$ is admissible, $u(t)$ \edit{with degree} $n\geq0$ is also admissible according to~\cite{Villaverde2019input}.
\end{proof}

However, Theorem~\ref{theorem:oic} provides a general result that is valid for all values of the state and parameters except for a set of measure zero. \newedit{It is possible that for some special initial conditions the matrix is not full rank. The symbolic calculator leads to generic conclusions on the rank of a matrix and fails to capture a specific set (of measure zero). Therefore, it is possible to categorize the system as locally structurally identifiable whereas for special initial conditions the result is uninformative.}
The following proposition gives an example of special initial condition that leads to unidentifiable CTH-RV system:
\begin{proposition}
\label{prop2}
The CTH-RV system~\eqref{eq:CTHRV_augmented} is unidentifiable given an equilibrium initial \newedit{condition~\eqref{eq:ics_2}} and constant input $u(t)=u_0$.
\end{proposition}


\begin{proof}
We prove that the system~\eqref{eq:CTHRV_augmented} is unidentifiable under constant input $u(t) = u_0$ with an equilibrium initial conditions specified as
\begin{equation}
    \label{eq:ics_2}
    x(0) = x_0^* = \begin{bmatrix}
    \tau u_0\\
     u_0
    \end{bmatrix}.
\end{equation}
where $u_0$ is the initial value of the lead vehicle velocity. The equilibrium initial condition results in $\dot{s}(0)=0$ and $\dot{v}(0)=f_{\text{CTHRV}}(u_0, x_0^*, \theta)=0$. The same four Lie Derivatives~\eqref{eq:lie_derivatives_proof1} will be used to calculate $\mathcal{O}_I (\tilde{x}_0,u)$. However, due to the specified initial conditions and the constant input condition, $\mathcal{O}_I (\tilde{x}_0,u)$ becomes:
\begin{equation}
\begin{aligned}
    &\mathcal{O}_I (\tilde{x}_0,u) = \\
    &\begin{bmatrix}
    1 & 0 & 0 & 0 & 0\\
    0 &-1 & 0 & 0 & 0\\
    -k_1 & k_2+k_1\tau & 0 & 0 & k_1 u_0\\
    k_1(k_2+k_1\tau) & k_1-(k_2+k_1\tau)^2 &0& 0& -k_1 u_0(k_2+k_1\tau)&\\
    o_{51}& o_{52}& 0&0& o_{55}&
    \end{bmatrix}
    \end{aligned}
\end{equation}
\begingroup
\allowdisplaybreaks
\begin{align*}
o_{51} &=k_1 (k_1 - (k_2 + k_1 \tau)^2)\\
o_{52} &=- k_1 (k_2 + k_1 \tau) - (k_2 + k_1 \tau) (k_1 - (k_2 + k_1 \tau)^2)\\
o_{55} &= -k_1u_0(k_1-(k_2+k_1\tau)^2).
\end{align*}
\endgroup
It is obvious that rank$(\mathcal{O}_I (\tilde{x}_0,u))=3$, and therefore this system is structurally unidentifiable under equilibrium initial conditions. Furthermore, the specific unidentifiable parameter(s) can be detected by removing each of the columns of $\mathcal{O}_I (\tilde{x}_0,u)$ and recalculating the rank. If the rank does not change after removing the $i$th column, than the corresponding $i$th variable in the augmented state $\tilde{x}$ is unidentifiable (or observable)~\cite{villaverde2019full}. It is straightforward to see that removing the 3rd or the 4th column of $\mathcal{O}_I (\tilde{x}_0,u)$ does not change the rank, and therefore, the corresponding parameters $k_1$ and $k_2$ are unidentifiable.
\end{proof}

\begin{corollary}
 Given an equilibrium initial condition $x_0^*$ for the CTH-RV system~\eqref{eq:CTHRV_augmented}, the admissible input required to enable structural identifiability is \edit{with degree} $n\geq 1$.
\end{corollary}

In the case of equilibrium initial condition, a constant input is not admissible to enable structural identifiability. Therefore, time-varying input may be needed to excite the system. We check the order of the input required by replacing the higher order derivatives of $u(t)$ in $\mathcal{O}_I$ with zero and recalculating rank$(\mathcal{O}_I)$. This procedure gives a full rank $(\mathcal{O}_I)$ when $\ddot{u}=0$, as long as $\dot{u}\neq 0$. Therefore an input \edit{with degree} $n\geq 1$ is admissible for structural identifiability of CTH-RV, even \edit{if} the system starts from an equilibrium initial condition.

Another unidentifiable case that \textit{differential geometry} fails to detect constitutes a special relationship between the initial condition and the parameters. As long as such relationship is established, there exists no admissible input, and the initial conditions need not to be at equilibrium. This case is demonstrated below.
\begin{proposition}
\label{prop3}
Given a generic initial condition $x_0=[s_0, v_0]^T$ for the CTH-RV system~\eqref{eq:CTHRV_augmented}, if $\tau=\frac{s_0}{v_0}$ and $k_2=\frac{v_0}{s_0}$, no admissible input exists.
\end{proposition}
\begin{proof}
Please see the Appendix~\ref{appendix:proof_prop3} for details.
\end{proof}


\begin{remark}
Measuring the space gap alone $y(t) = s(t)$ leads to the same identifiability results as measuring both states $y(t) = [s(t), v(t)]^T$. One can arrive to this result by checking the rank of the observability-identifiability matrix $\mathcal{O}_I (\tilde{x}_0,u)$. This finding is in agreement with~\cite{PunzoModelCalibration2005,TREIBER2013922} that using space gap instead of velocity or acceleration profile to calibrate a car-following model leads to sensitivity in the objective function.
\end{remark}

\paragraph{Visualization}
As an illustration of the above analysis for CTH-RV, we visualize the output differences $e(\cdot,\cdot)$ corresponding the settings of Propositions~\ref{prop1}-\ref{prop3}. We first fix one parameter as $\theta_\text{true}$, and compare the output error under another parameter $\theta$. Let $e(\theta,\theta_\text{true})$ be defined as the \textit{mean-squared-error} (MSE):
\begin{equation}
\label{eq:cost_function}
    e(\theta, \theta_\text{true})=
    \int_0^\mathcal{T}\lVert{y}(t,\theta_{\text{true}})-{y}(t,\theta)\rVert_2^2dt,
\end{equation}
which describes the error between the true output simulated by $\theta_{\text{true}}$ and the output simulated by $\theta$, under the same initial conditions and input. 
Let $\mathcal{T}$ denote the total time over which the ODE is solved. 

We visualize $e(\theta,\theta_\text{true})$ as one sweeps over the parameter space for CTH-RV (see Figure~\ref{fig:contour_CTHRV_1}-\ref{fig:contour_CTHRV_3}). Figure~\ref{fig:contour_CTHRV_1} shows that the error is zero only at $\theta=\theta_\text{true}$, indicating an identifiable system. Figure~\ref{fig:contour_CTHRV_2} and \ref{fig:contour_CTHRV_3}, on the contrary, show that even for a structurally locally identifiable system, there exists an initial condition, parameters, and input combinations such that distinct parameters exist and produce the same output, as suggested by Proposition~\ref{prop2} and \ref{prop3}.  Specifically, we see in Figure~\ref{fig:contour_CTHRV_2} that $k_1$ and $k_2$ are unidentifiable under an equilibrium initial condition and a constant input. This is consistent with the finding from~\cite{wang2020estimating}; in Figure~\ref{fig:contour_CTHRV_3}, we see that that $k_1$ is unidentifiable when the initial speed and space gap satisfies $s_0=\tau v_0$, even though the initial condition is not at a equilibrium state (i.e., $v_0\neq u_0$). 


\begin{figure*}
\begin{minipage}{\textwidth} 
    \includegraphics[width=\linewidth]{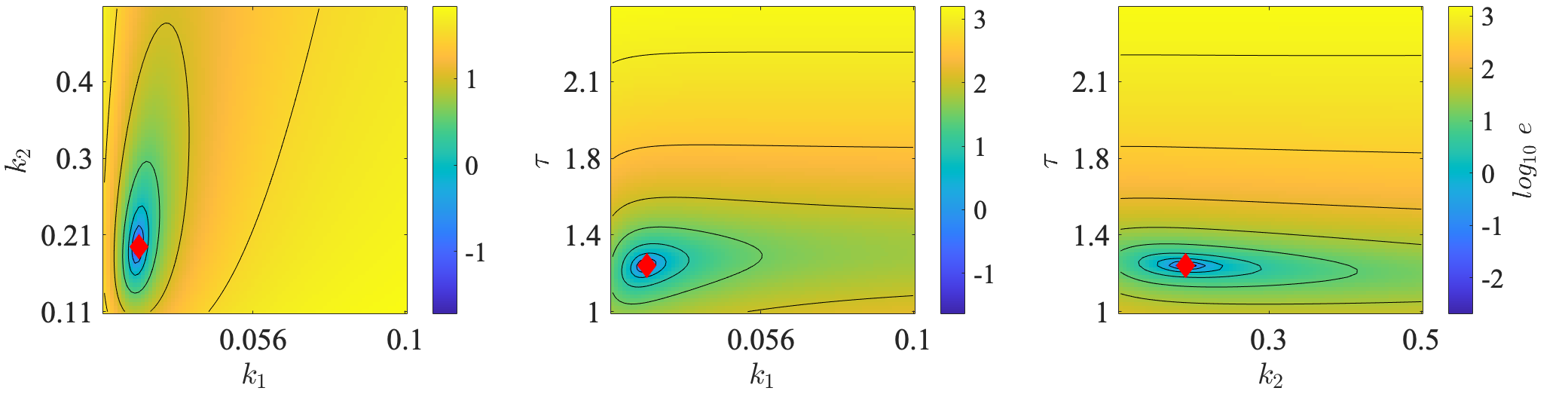}
    \subcaption{An identifiable scenario: initial condition $x_0=[60,20]^T$, input $u(t)=31$. $\theta_{\text{true}}=[k_1, k_2, \tau]^T=[0.0216, 0.1943, 1.2293]^T$.}
    \label{fig:contour_CTHRV_1}
\end{minipage}    
\hspace{\fill}  
\begin{minipage}{\textwidth} 
    \includegraphics[width=\linewidth]{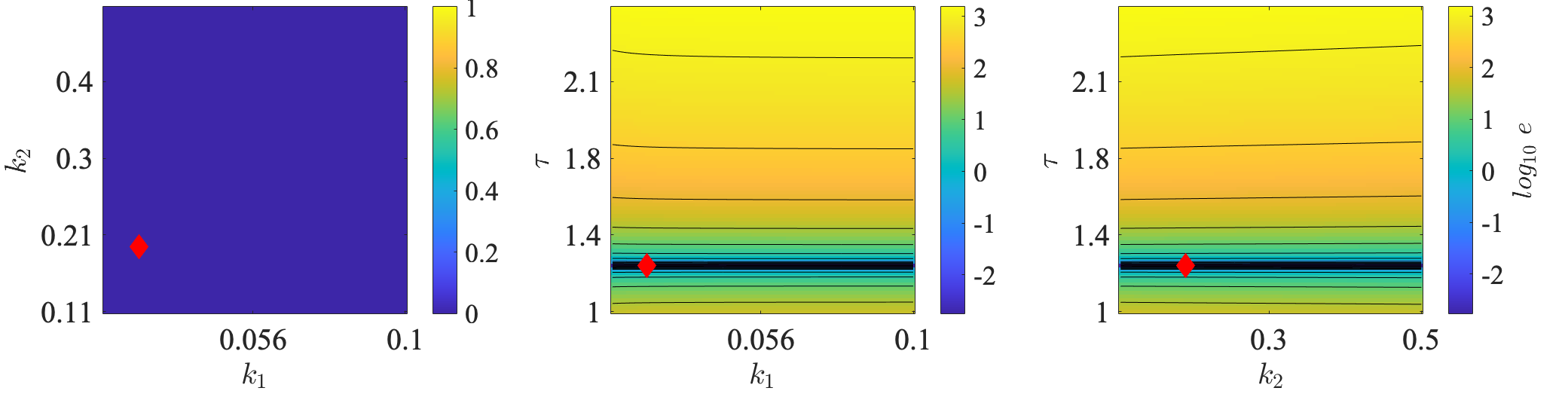}
    \subcaption{An unidentifiable scenario: equilibrium initial condition $x_0^*=[38.1, 31]^T$ and $u(t)=31$. $\theta_{\text{true}}=[k_1, k_2, \tau]^T=[0.0216, 0.1943, 1.2293]^T$.}
     \label{fig:contour_CTHRV_2}
\end{minipage}  
\hspace{\fill}  
\begin{minipage}{\textwidth} 
    \includegraphics[width=\linewidth]{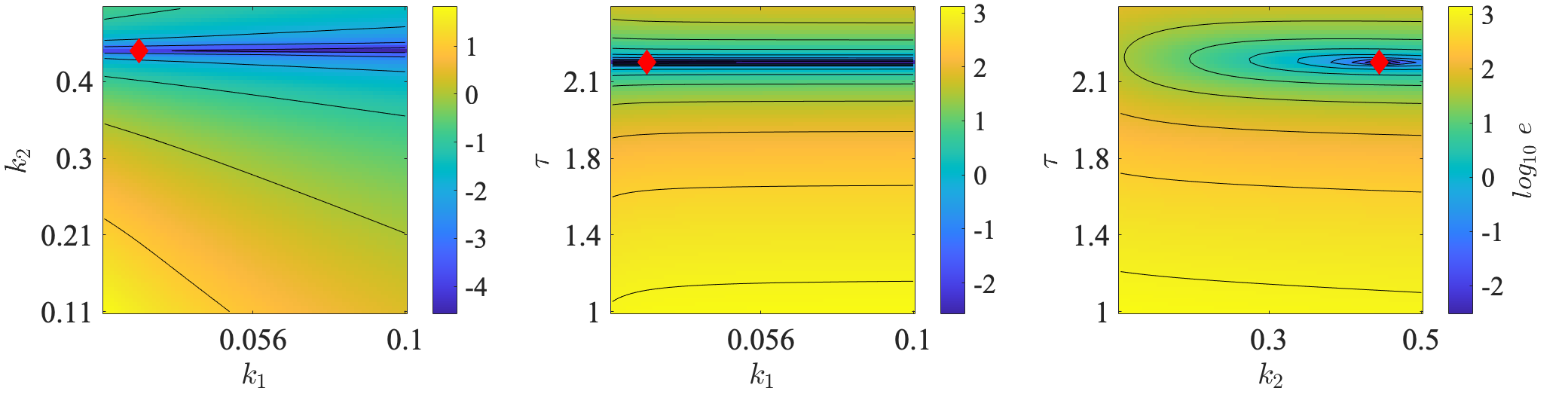}
    \subcaption{An unidentifiable scenario: initial condition $x_0=[72.7, 32.5]^T$ and $u(t)$ is shown in Figure~\ref{fig:lead_vel}. $\theta_{\text{true}}=[k_1, k_2, \tau]^T=[0.0216, 0.4472, 2.2361]^T$.}
     \label{fig:contour_CTHRV_3}
\end{minipage}  

\caption{Visualization of $e(\theta,\theta_{\text{true}})$ for CTH-RV. Red diamond indicates $\theta_{\text{true}}$. Even for structurally locally identifiable systems, there may exist an initial condition and input combination such that certain parameter become unidentifiable.
}
\label{fig:contour_CTHRV}
\end{figure*}

\subsection{Structural local identifiability analysis for other models}
Following the same procedure as was used to analyze the CTH-RV, we check the observability-identifiability matrix $\mathcal{O}_I (\tilde{x}_0,u)$ for the other three models, i.e., the OV model~\eqref{eq:bando}, the FTL model~\eqref{eq:FTL} and the IDM~\eqref{eq:Enhanced_IDM}. The structural identifiability results for each of the models under almost all initial conditions and the corresponding \textit{a priori} set of admissible inputs are summarized in Table~\ref{tab:id_result_generic_IC}. For specific initial conditions (e.g., an equilibrium initial condition), the system may require a higher-order input to enable structural identifiability (also see Table~\ref{tab:id_result_generic_IC}).

\begin{table}
\caption{Structural local identifiability summary.}
\centering
\begin{threeparttable}
\begin{tabular}{@{}cccc@{}}
\toprule
Model            & Parameters              & \multicolumn{2}{c}{\edit{Admissible polynomial} input $u(t)$ \edit{degree $n$}}  \\
\cline{3-4}
$f_{\text{CF}}$  & $\theta$                &  Generic (almost all) $x_0$ &  Equilibrium $x_0^*\dag$\\
\midrule
\multirow{1}{*}{CTH-RV} & \multirow{1}{*}{$k_1$, $k_2$, $\tau $} & $n\geq0$  & $n\geq1$\\
\multirow{1}{*}{OV}     & \multirow{1}{*}{$\alpha$, $a$, $b$, $h_m$} & $n\geq0$  & $n\geq1$\\
\multirow{1}{*}{FTL}    & \multirow{1}{*}{$C$, $\gamma$}  & $n\geq0$ & N/A\\
\multirow{1}{*}{IDM}    & \multirow{1}{*}{$a$, $b$, $s_j$, $v_f$, $T$} &$n\geq0$ & $n\geq1$\\
\bottomrule
\end{tabular}
\footnotesize
\begin{tablenotes}
\item[\dag] The equilibrium initial condition for each model is listed in Section~\ref{sec:preliminaries}.
\end{tablenotes}
\end{threeparttable}
\label{tab:id_result_generic_IC}
\end{table}

Table~\ref{tab:id_result_generic_IC} suggests that all the models are structurally locally identifiable, i.e., parameters for each model can be uniquely identified under any admissible input given almost \newedit{all} initial conditions, whether it is known or not, even if only space-gap is being measured. Moreover, a constant input \edit{degree $n=0$} is an admissible input to enable structural identifiability of all the models. Additionally, all inputs $u(t)$ \edit{with degree} $n\geq0$ are also admissible to enable structural identifiability. 

However, the analysis does not preclude the existence of an initial condition (e.g., an equilibrium initial condition) for which the models become unidentifiable for a \edit{constant} input. For some models, higher-order inputs are required to enable structural identifiability. For CTH-RV, OV and IDM, any input \edit{with degree} $n\geq1$ is an admissible input to enable identifiablity for the equilibrium initial condition. On the other hand, the FTL completely loses identifiability given an equilibrium initial condition $v_0=u_0$, since no admissible inputs can be found to enable identifiability. 

Overall, the main outcome of the structural local identifiability analysis is that we find no  fundamental intrinsic structural problems in any of the analyzed models. This means it is possible to design an experiment (i.e., chose an initial condition and an input) in order to uniquely infer the unknown model parameters.

The structural local identifiability analysis provides theoretical results that are generally valid for almost all numerical values of the initial conditions and the parameters, but the test could fail to detect unidentifiable scenarios (even for a simple CTH-RV model, it fails to detect a special initial condition (Proposition~\ref{prop2}) or parameter dependency (Proposition~\ref{prop3}) that causes non-identifiability). 
\edit{In practice, identifiability would be impossible when close to these special conditions (because of the rank deficiency of the observability-identifiability matrix). The probability of being in a small open set around these measure-zero sets, although unlikely, is not zero. Therefore}, it is important to ask a different question in practice: \textit{given the experimental design (i.e., given a specific initial condition and input trajectory that are known), whether there exists multiple distinct parameters that would produce the same output}. If multiple distinct parameters are indistinguishable in the output space, then the corresponding model(s) are unidentifiable. We proceed to numerical methods in the next section to address this question.

\section{Practical identifiability analysis}
\label{sec:practical_id_results}
\begin{table}
\centering
\caption{Parameter bounds.}
{\begin{tabular}{cccc} \toprule
 Models &  Parameter $\theta$& Lower bound  $\theta_{\textbf{min}}$& Upper bound $\theta_{\textbf{max}}$  \\ 
 \midrule
CTHRV
& $[k_1, k_2, \tau_h]$ 
& [0.001, 0.01, 0.1]
& [1, 1, 3]\\
OV   
& $[\alpha, a, h_m, b]$  
& [0.5, 10, 2, 18]  
& [3.3, 32, 30, 45]\\
  
FTL  &  $[C, \gamma]$   & [100, 1]    &[600, 3]\\
IDM 
& $[s_j, v_f, T, a,b]$    
& [3,  21,  0.1,  0.1,   0.5]   
& [25,  41,   3,   3,  5]
                      \\ \bottomrule
\end{tabular}}
\label{tab:param_bounds}
\end{table}
Although the initial condition and parameter subspace that loses structural identifiability for CTH-RV can be analytically solved for, it is much challenging to exhaustively find the special cases for other models composed of irrational functions. To this end, we use a numerical method to analyze a more specific notion of identifiability: \textit{given a specific initial and input condition that are known and non-trivial, i.e., non-equilibrium initial condition and time-varying input, are there distinct parameters that produce the same or similar output?} If such distinct parameter sets are found, then the corresponding model is unidentifiable to this specific experimental setup, irrespective of the cause (initial condition or parameters fall into a measure zero set or the input is not informative enough). In this section, we present identifiability analysis using \textit{direct test} for models~\eqref{eq:CTH-RV}, \eqref{eq:bando}, \eqref{eq:FTL} and \eqref{eq:Enhanced_IDM}. First, we provide the details of the numerical experiment setup, including a specific optimization formulation. Next, we solve~\eqref{eq:opt2} with 1) the tightest constraint $\epsilon=10^{-6}$ and 2) the relaxed constraint $\epsilon>10^{-6}$. Because the choice of the cut-off threshold $\epsilon$ is arbitrary, we incrementally increase $\epsilon$ to observe the sensitivity of model calibration to the measurement error.



For $\delta$-identifiability at both $\epsilon=10^{-6}$ and $\epsilon>10^{-6}$, we first provide the detailed analysis on the CTH-RV model~\eqref{eq:CTH-RV}, and summarize the results for the others. 

\subsection{Implementation details}
Consider the following optimization problem formulation for~\eqref{eq:opt2}:
\newedit{
\begin{equation}
\label{eq:opt3}
\begin{aligned}
& \underset{\theta_{1,i},\theta_{2,i} \in [\theta_{\text{min},i},\theta_{\text{max},i}], \forall i}{\text{maximize}}
& & d(\theta_1, \theta_2) = \frac{1}{\sqrt{n_{\theta}}}\sqrt{\sum_{i=1}^{n_{\theta}}\left( \frac{\theta_{1,i}-\theta_{2,i}}{\theta_{\text{max},i}-\theta_{\text{min},i}}\right)^2} \\
& \text{subject to}
& & e(\theta_1, \theta_2)=\frac{1}{K}{\sum_{k=0}^{k=K} \left\lVert y_k(\theta_1)-y_k(\theta_2)\right\rVert^2_2}\leq\epsilon.
\end{aligned}
\end{equation}
}
The subscript $i$ on $\theta$ denotes the $i^{\text{th}}$ element in the parameter vector. The resulting distance (objective function) is a normalized Euclidean distance scalar between 0 and 1 bounded by $\theta_{\text{min}}$ and $\theta_{\text{max}}$ (Table~\ref{tab:param_bounds}), i.e., $d(\theta_{\text{min}},\theta_{\text{max}})=1$. The trajectory difference $e(\cdot, \cdot)$ in the constraint is a mean squared error to measure the difference between $y_k(\theta_1)$ and $y_k(\theta_2)$, i.e., the discrete-time space-gap simulated with an Euler method using $\theta_1$, $\theta_2$, respectively. The number of decision variables is $2\times n_{\theta}$ where $n_{\theta}$ is the dimension of the parameters for each model. Recall that we denote the solution to~\eqref{eq:opt2} as $\delta^*$ and the corresponding parameter pair as $(\theta_1^*,\theta_2^*)$.

 The optimization solver used to solve~\eqref{eq:opt3} is \texttt{patternsearch} in the MATLAB global optimization toolbox~\cite{matlabGOTB}. The \texttt{patternsearch} is a gradient-free optimization solver that works with potentially non-smooth objective functions. More details regarding the solver can be found in~\cite{charles2002pattern}.
 

\begin{figure}
    \centering
    \includegraphics[width=0.6\linewidth]{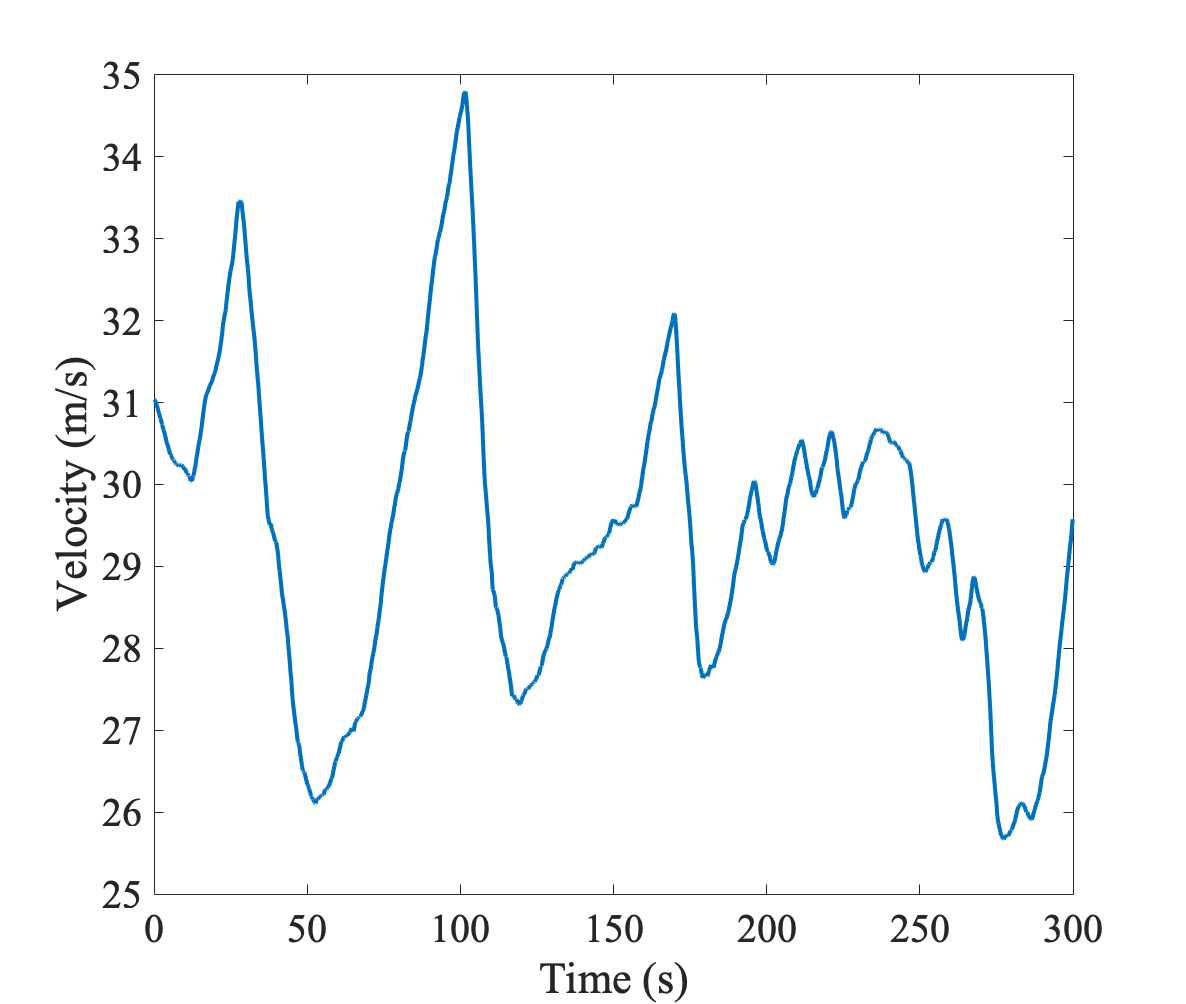}
    \caption{A time-varying lead vehicle velocity profile.}
    \label{fig:lead_vel}
\end{figure}

\subsection{Practical identifiability with error-free measurement}
In the numerical setting, we specify the tightest constraint at approximately zero (specifically, $\epsilon=10^{-6}$) and solve for the optimization problem~\eqref{eq:opt3} with a time-varying input profile (Figure~\ref{fig:lead_vel}) and a fixed, non-equilibrium initial condition $x_0=[72.7, 32.5]^T$. We use this initial condition and input for all of the models in the analysis for the remaining of this paper. Given this specific experiment setup, direct test aims to find indistinguishable (thus unidentifiable) parameters that separate farthest in the parameter space. This specific initial condition is chosen to illustrate the utility of the direct test, i.e. given an arbitrary but non-equilibrium initial condition, can we find two distinct parameters that generate the same output.

\paragraph{CTH-RV}
\begin{table}
\caption{Direct test finds two indistinguishable parameter sets for CTH-RV with $\epsilon=$1e-6}
\centering
\begin{tabular}{@{}ccrrr@{}}
\toprule

Model & Parameter & $\theta_1^*$ & $\theta_2^*$ & $\delta^*$ 
\\
 \midrule
\multirow{3}{*}{CTHRV} & $k_1$ & 1 & 0.001 \\
                        & $k_2$ & 0.4472 & 0.4472 & 0.5774 \\
                        & $\tau$ & 2.236 & 2.236
                        \\ 
\bottomrule
\end{tabular}
\label{tab:rho0_CTHRV}
\end{table}

\begin{figure}
    \centering
    \includegraphics[width=\linewidth]{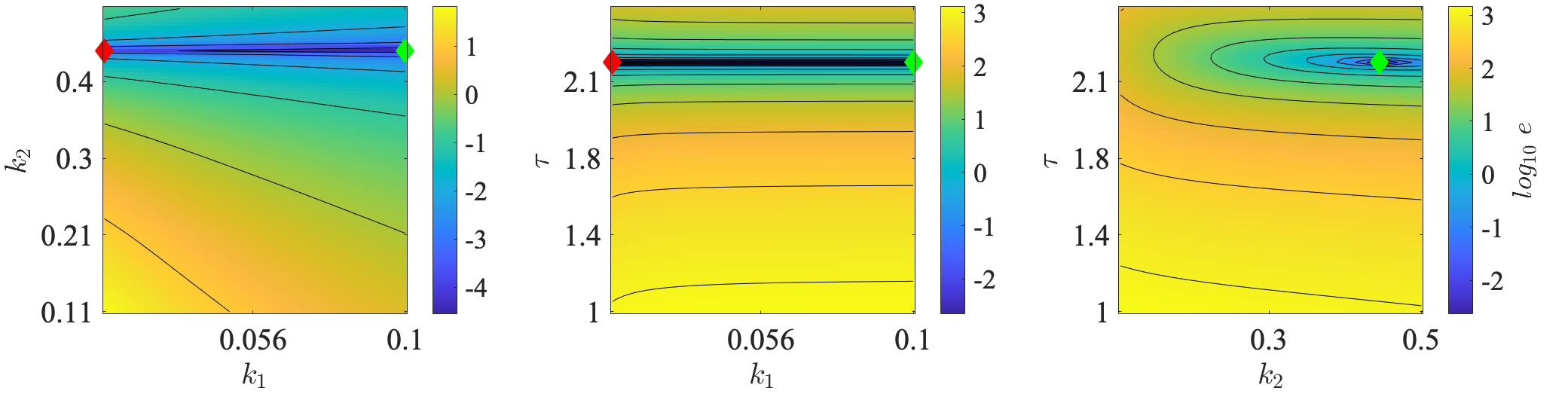}
    \caption{Visualize indistinguishable parameter sets $\theta_1^*$ (red diamond) and $\theta_2^*$ (green diamond) from Table~\ref{tab:rho0_CTHRV}.}
    \label{fig:contour_CTHRV_dual}
\end{figure}

The direct test result for CTH-RV by solving the optimization problem with $\epsilon=10^{-6}$ is listed in Table~\ref{tab:rho0_CTHRV}.  Table~\ref{tab:rho0_CTHRV} suggests that given a non-equilibrium lead velocity profile (Figure~\ref{fig:lead_vel}) and an arbitrary, non-equilibrium initial condition that are known, the CTH-RV is unidentifiable. Solving problem~\eqref{eq:opt3} results in $\delta^*=0.577$, meaning there exists two distinct parameters that generate the same output. Figure~\ref{fig:contour_CTHRV_dual} visualizes this result: $\theta_1^*$ and $\theta_2^*$ are far apart but produce the same output. The value of $e(\theta, \theta_1^*)$ remains constant along the unidentifiable parameter. Furthermore, the unidentifiable parameters fall into a set such that $\tau=s_0/v_0$ and $k_2=v_0/s_0$. This corresponds to Proposition~\ref{prop3} where $k_1$ cannot be uniquely identified. 

\edit{Additionally, we discovered that when the initial conditions are in a neighborhood of the measure-zero set identified in Proposition~\ref{prop3}, CTH-RV can still be practically unidentifiable, as shown in Figure~\ref{fig:CTHRV_indistinguishable}. In this example, the initial condition is $s_0 = 72$, $v_0=32$, and the parameters are $\theta_1=[1,0.394,2.35]^T$ and $\theta_2=[0.1,0.394,2.35]^T$. A small output difference is obtained ($\epsilon=0.0745$) with distinct parameters ($\delta^*=0.5201$). Note that $k_2\approx \dfrac{v_0}{s_0}$ and $\tau \approx \dfrac{s_0}{v_0}$, which mean that the initial conditions are not exactly in the measure-zero set as shown in Proposition~\ref{prop3}, but in a neighborhood of a it.}
\begin{figure}
    \centering
    \includegraphics[scale=0.2]{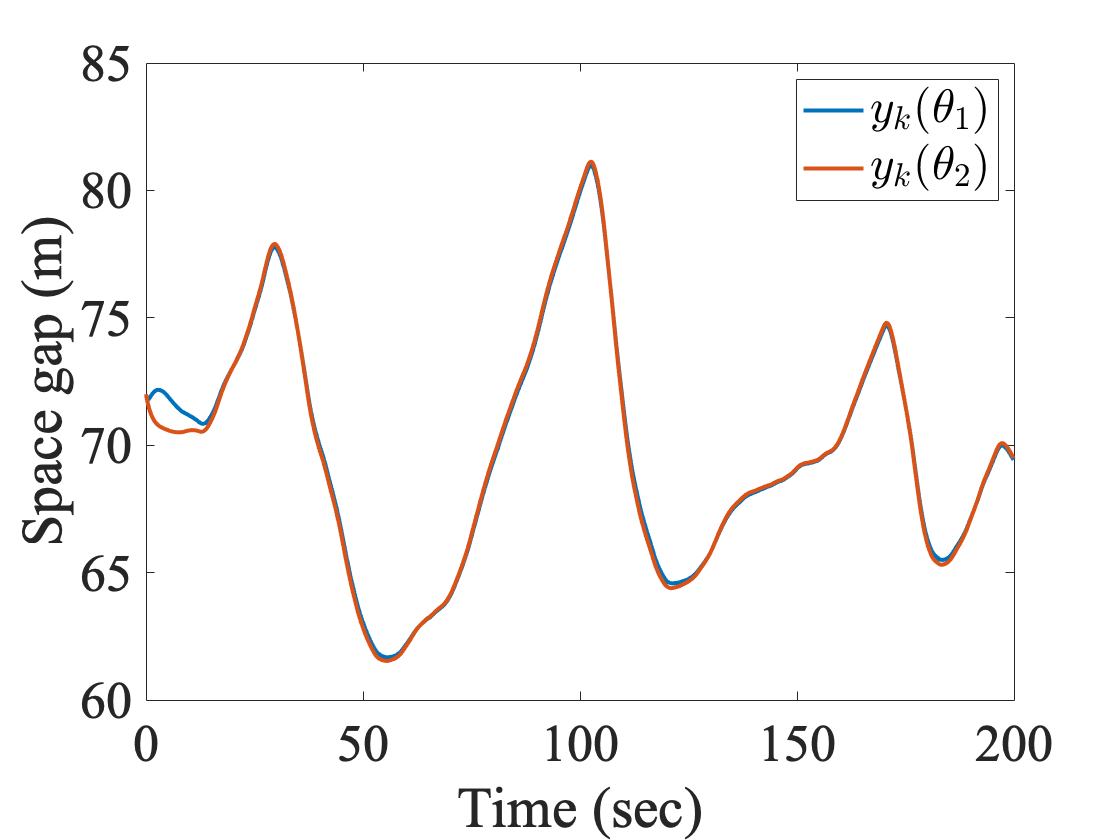}
    \caption{Visualize practically indistinguishable parameter sets for CTHRV.}
    \label{fig:CTHRV_indistinguishable}
\end{figure}
\paragraph{Other models}
\begin{table}
\caption{Direct test finds indistinguishable parameter sets for other models with $\epsilon=$1e-6}
\centering
\begin{tabular}{@{}ccccc@{}}
\toprule

Model & Parameter & $\theta_1^*$ & $\theta_2^*$ & $\delta^*$
\\
 \midrule

\multirow{4}{*}{OV}  & $\alpha$ & 3.0772 & 3.1427 &\multirow{4}{*}{0.0117} \\
                        & $a$ &  19.7485 & 19.7451 & \\
                        & $h_m$ & 22.2094 & 22.2122 & \\
                        & $b$ & 23.2986 & 23.2915 &  \\
                        \midrule
\multirow{2}{*}{FTL}  &  $C$ & 130.0285 & 599.9699 &\multirow{2}{*}{0.6766}\\
                    & $\gamma$ & 1 & 1.3582 & \\
                    \midrule
\multirow{5}{*}{IDM}  & $s_0$ & 10.5615 & 8.3913 &\multirow{5}{*}{0.0476}\\
                    & $v_0$ & 35.788 & 35.771 &\\
                    & $T$ & 2.787 & 2.903 &\\
                    & $a$ & 2.559 & 2.494 &\\
                    & $b$ & 3.395 & 3.395 & \\\bottomrule
\end{tabular}
\label{tab:rho0}
\end{table}
As for the other three models, the solution to the direct test~\eqref{eq:opt3} is summarized in Table~\ref{tab:rho0}. The results indicate that the FTL~\eqref{eq:FTL} is also unidentifiable under the specific experiment setup ($\delta^*=0.677$). OV~\eqref{eq:bando} and IDM~\eqref{eq:Enhanced_IDM} are practically identifiable ($\delta^*=0.0117$ and $0.0476$, respectively, both are small), which suggests that only parameters that are close to each other can produce the same output. The solutions found by numerical direct test are visualized in Figures~\ref{fig:contour_CTHRV_dual}-\ref{fig:contour_IDM_dual} for the corresponding models. 

\begin{figure}
    \centering
    \includegraphics[width=\linewidth]{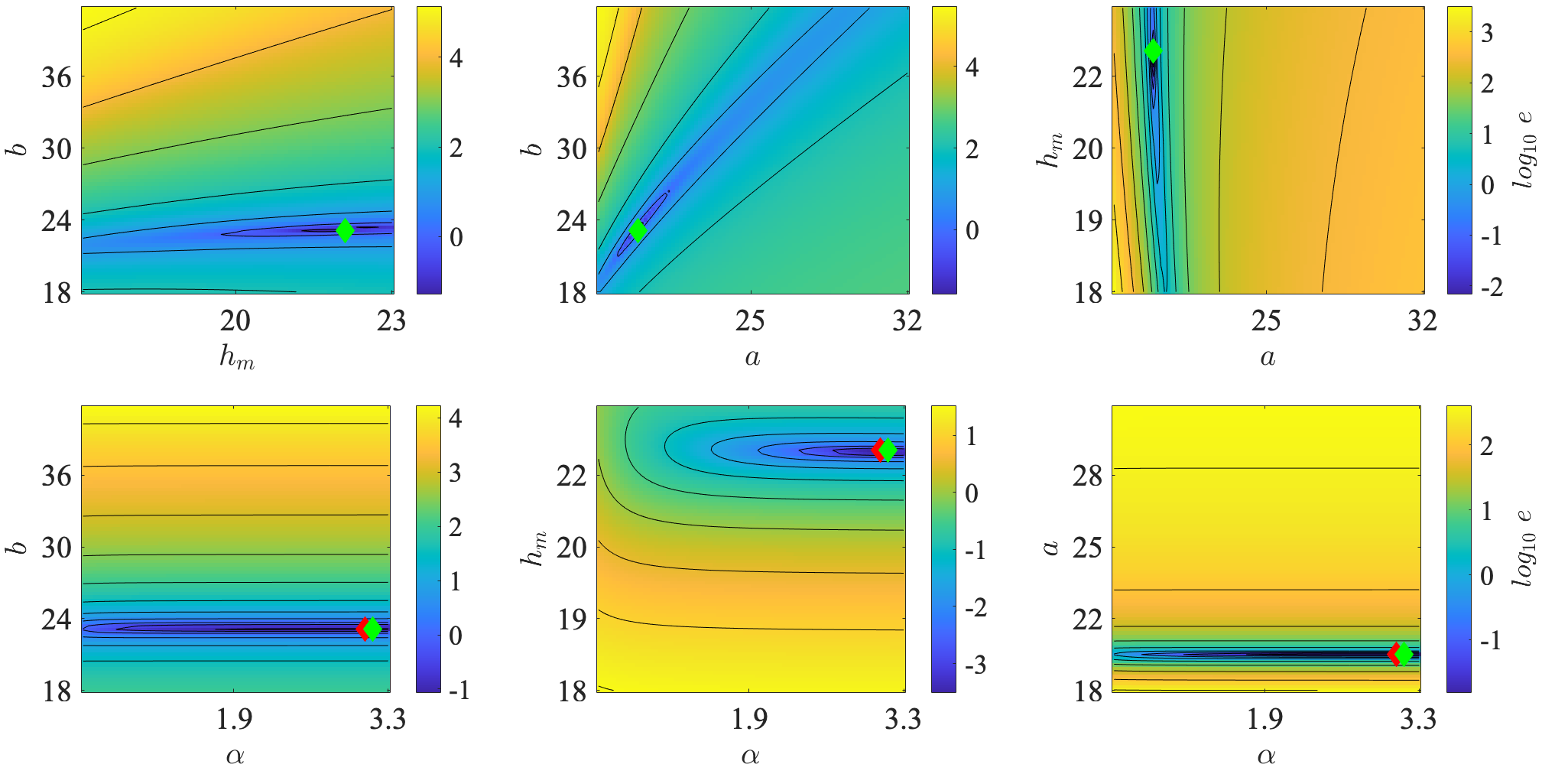}
    \caption{Visualize indistinguishable parameter sets $\theta_1^*$ (red diamond) and $\theta_2^*$ (green diamond) for OV from Table~\ref{tab:rho0}.}
    \label{fig:contour_BANDO_dual}
\end{figure}

\begin{figure}
    \centering
    \includegraphics[width=0.5\linewidth]{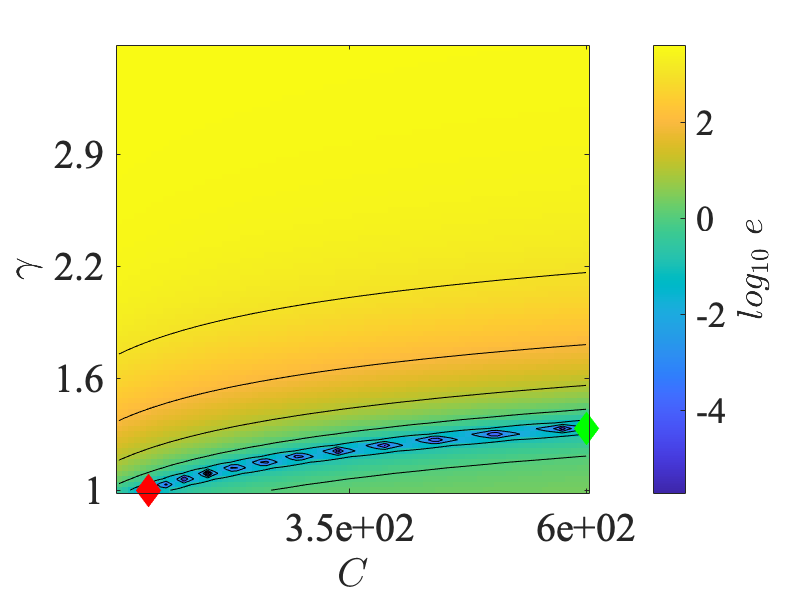}
    \caption{Visualize indistinguishable parameter sets $\theta_1^*$ (red diamond) and $\theta_2^*$ (green diamond) for FTL from Table~\ref{tab:rho0}.}
    \label{fig:contour_FTL_dual}
\end{figure}

\begin{figure}
    \centering
    \includegraphics[width=\linewidth]{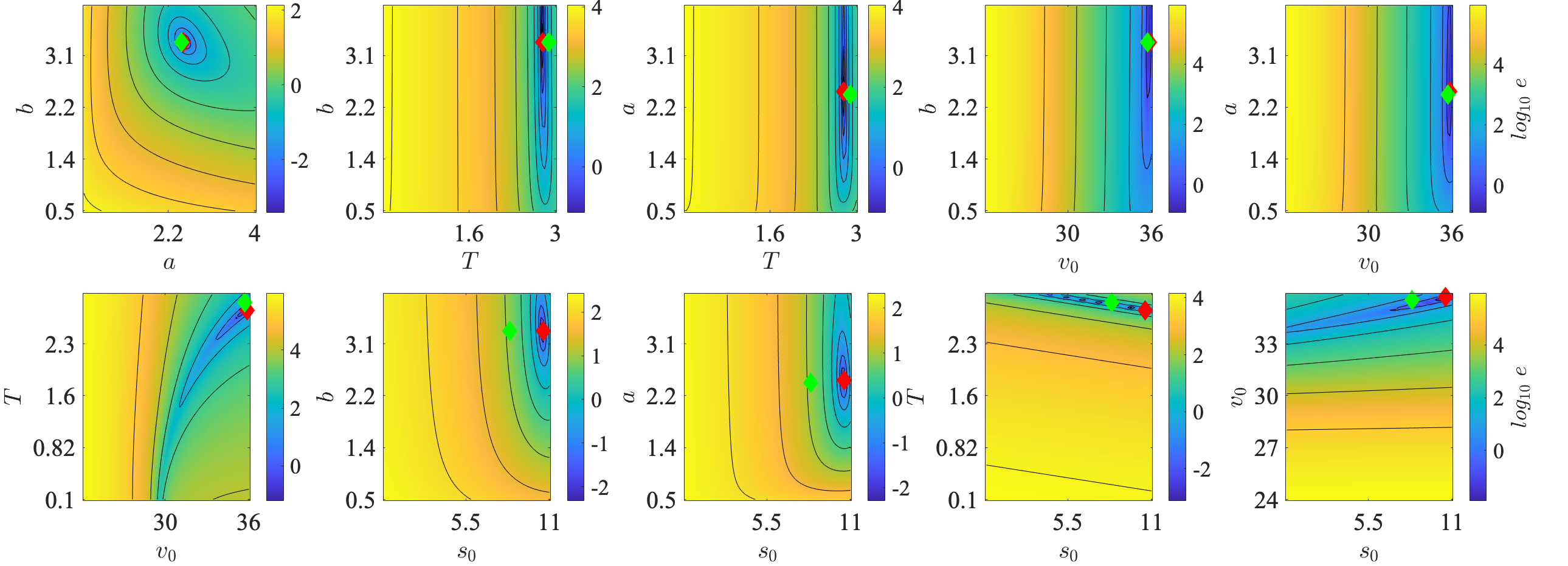}
    \caption{Visualize indistinguishable parameter sets $\theta_1^*$ (red diamond) and $\theta_2^*$ (green diamond) for IDM from Table~\ref{tab:rho0}.}
    \label{fig:contour_IDM_dual}
\end{figure}

\begin{figure*}
\begin{minipage}{0.49\textwidth} 
    \includegraphics[width=\linewidth]{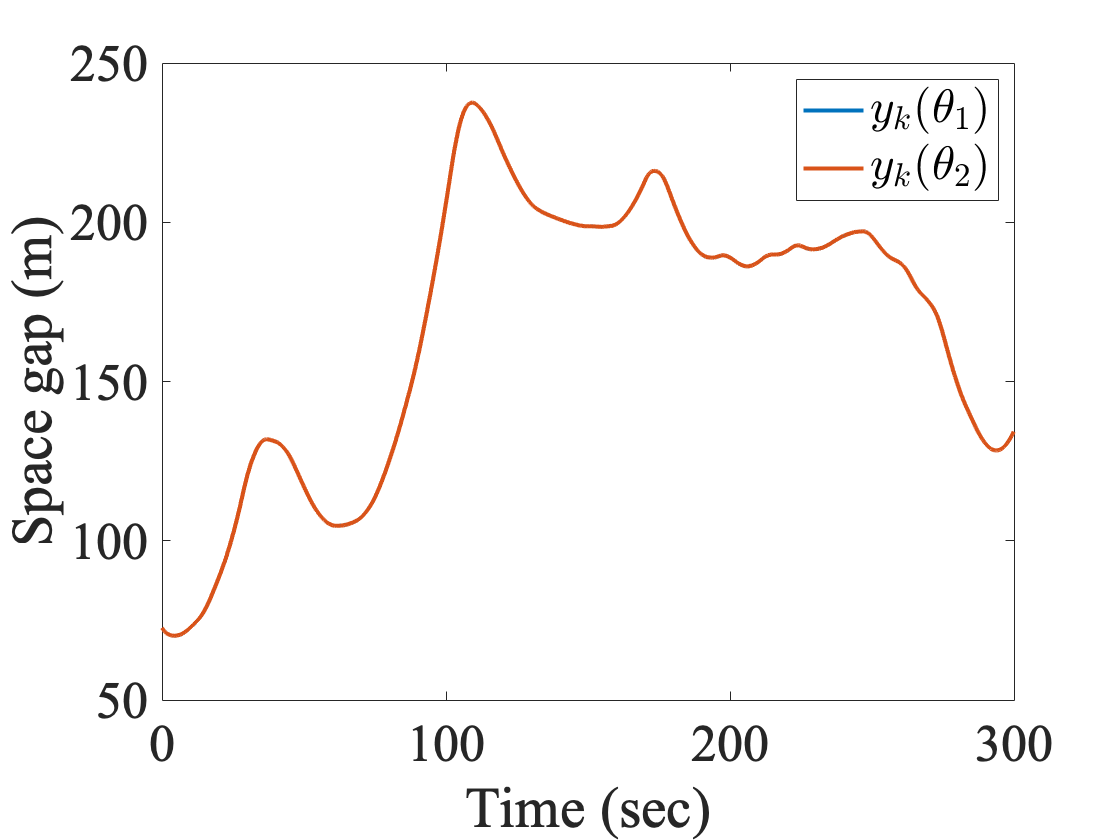}
    \subcaption{$e(\theta_1,\theta_2)=0$}    
    \label{fig:dual_IDM_0}
\end{minipage}    
\hspace{\fill}  
\begin{minipage}{0.49\textwidth} 
    \includegraphics[width=\linewidth]{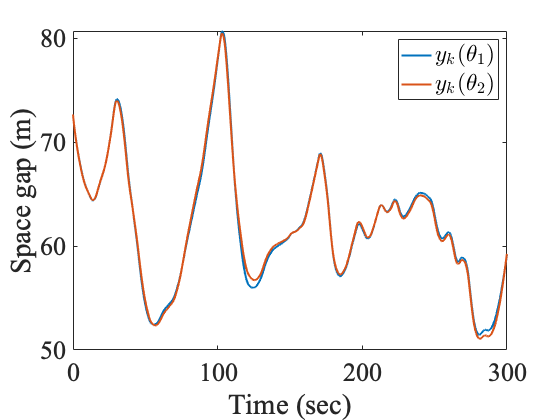}
    \subcaption{$e(\theta_1,\theta_2)=0.1$}
     \label{fig:dual_IDM_01}
\end{minipage}    

\begin{minipage}{0.49\textwidth} 
    \includegraphics[width=\linewidth]{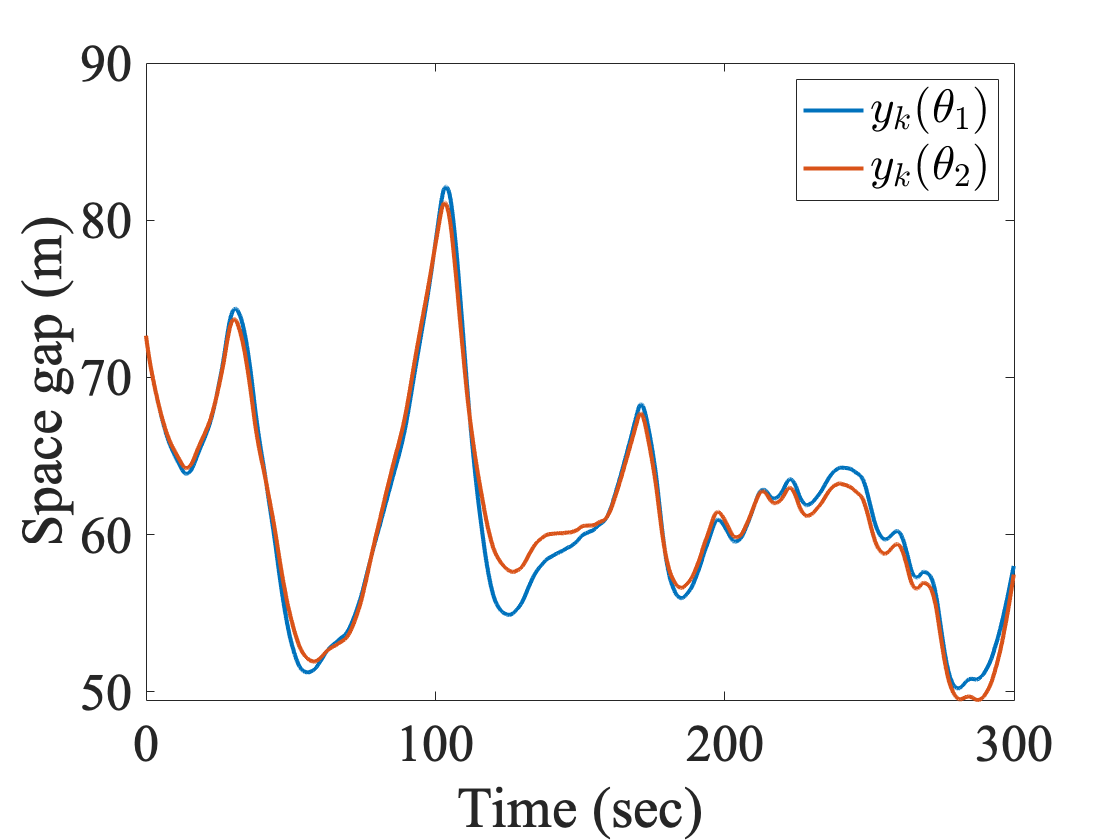}
    \subcaption{$e(\theta_1,\theta_2)=1$}
    \label{fig:dual_IDM_1}
\end{minipage}   
\hspace{\fill} 
\begin{minipage}{0.49\textwidth} 
    \includegraphics[width=\linewidth]{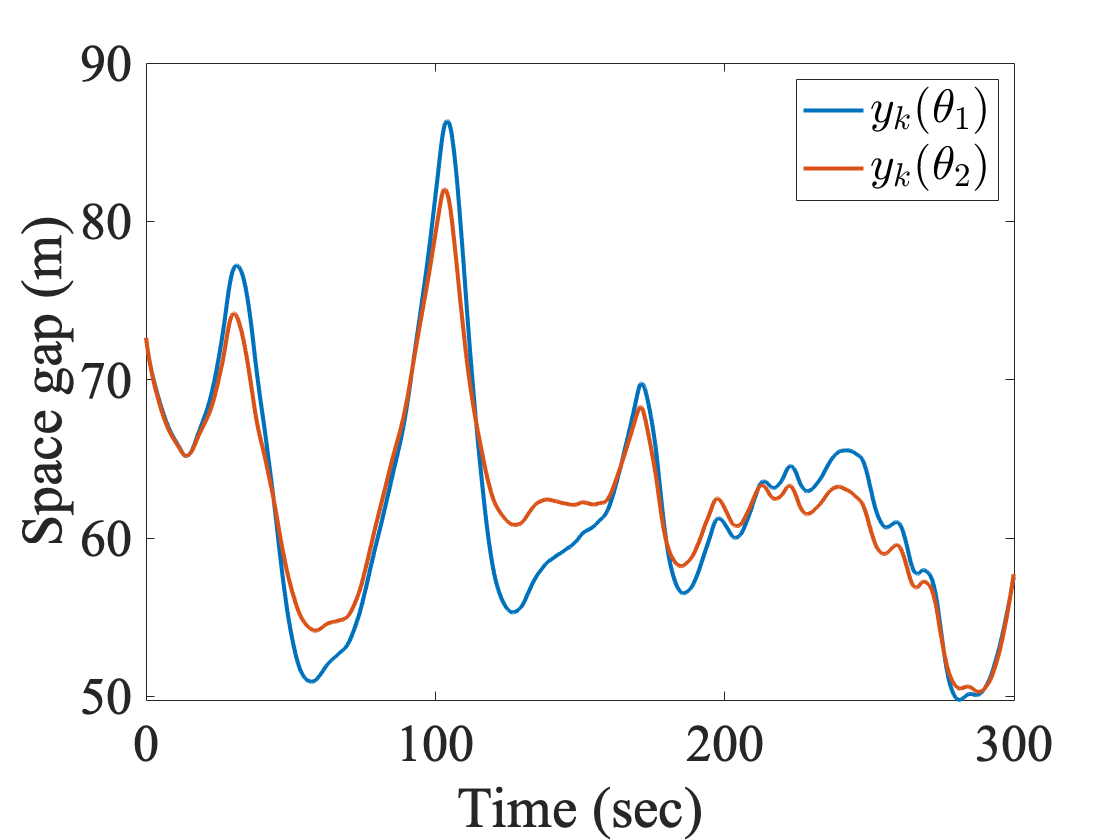}
    \subcaption{$e(\theta_1,\theta_2)=5$}
    \label{fig:dual_IDM_5}
\end{minipage} 
\caption{A demonstration of various trajectory differences $e(\theta_1,\theta_2)$ of IDM.}
\label{fig:dV_IDM_r}
\end{figure*}


\begin{figure}
    \centering
    \includegraphics[width=0.7\linewidth]{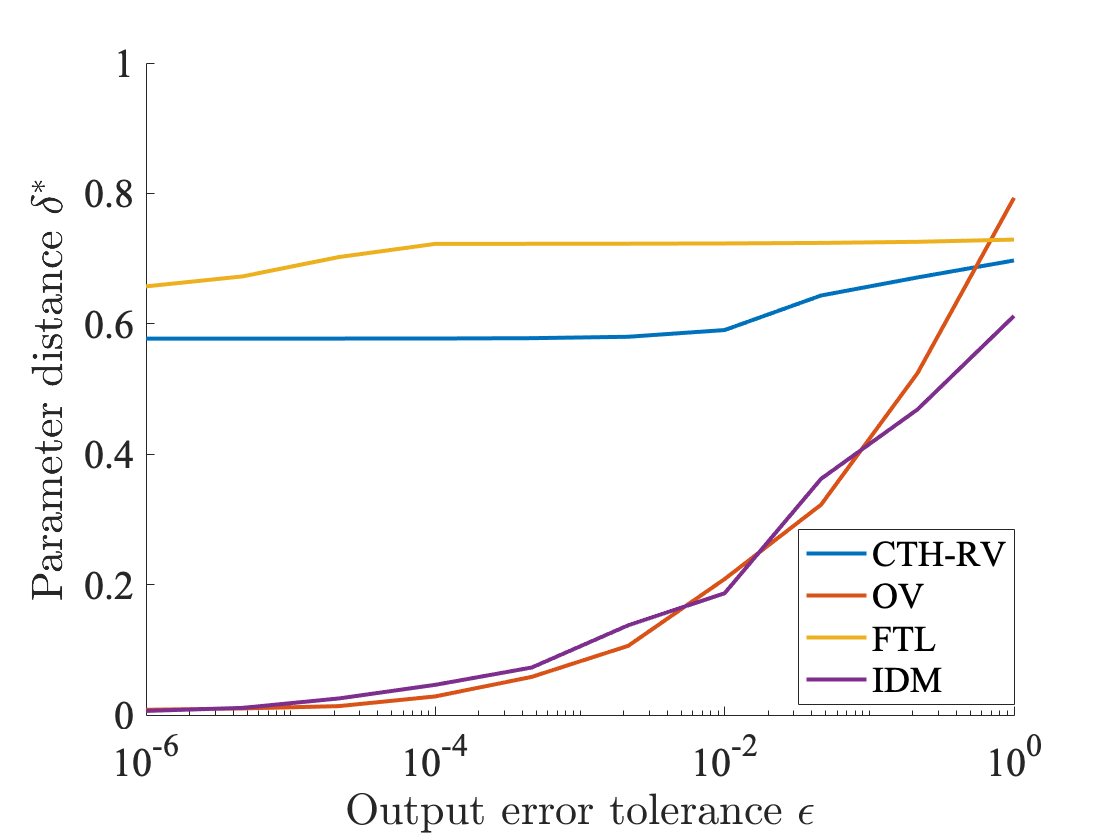}
    \caption{The sensitivity of practical identifiability with respect to the error threshold $\epsilon$. The specific setup is $x_0=[72.7, 32.5]^T$ and input $u(t)$ shown in Figure~\ref{fig:lead_vel} for all the models.}
    \label{fig:id_curve}
\end{figure}

\subsection{Practical identifiability when measurement noise is present}

\edit{The practical identifiability is nontrivial to quantify because 1) the threshold for $\delta^*$ to be identifiable/unidentifiable cannot be easily generalized across all models. Parameters for each model also have different physical meanings. 2) The trajectory error $\epsilon$ has multiple possible formulations (RMSE, ARE, MSE, etc.), making a universal cut-off threshold for practical identifiability impossible. Therefore, we provide the sensitivity analysis instead. The analysis is to provide a sense of robustness of parameter estimation with the presence of imperfect measurements (e.g., with measurement errors).}

The value of $\epsilon$ used in the constraint of the direct test problem is user-defined, and must be specified before solving the numerical problem. Therefore, we vary $\epsilon$ values to find the maximum parameter distances for each $\epsilon$ (see, for example, $e(\cdot,\cdot)=0, 0.1, 1$ and $5 \text{m}^2$ for IDM in Figure~\ref{fig:dV_IDM_r}). This is to analyze the sensitivity of the maximum distance between two indistinguishable parameters with respect to the tolerance on the output differences. Previous studies suggest that calibrating on empirical data achieves the lowest MSE of 1-2$\text{m}^2$ in space gap~\cite{gunter2019modelcomparison,wang2020estimating}, our analysis on practical identifiability will focus on $e(\cdot,\cdot)$ to be less than $1\text{m}^2$, or $\epsilon \in [0,1]$.

In particular, we select $\epsilon$  equally-spaced in log space between $10^{-6}$ and $10^{0}$, and solve for $\delta^*$ and the corresponding $(\theta_1^*,\theta_2^*)$ in problem~\eqref{eq:opt3}. The resulting curves shown in Figure~\ref{fig:id_curve} are the solutions $\delta^*$ as a function of $\epsilon$ in the constraint for each model. The curve is the maximum distance in parameter space such that the output error does not exceed $\epsilon$. Clearly, as the output difference $\epsilon$ increases, so does the maximum parameter distance $\delta^*$. Intuitively, given a specific initial condition and lead vehicle trajectory, as measurement error increases, the best-fit parameters are more likely to be non-unique, and the estimated parameters are less robust in the worst case. 

Furthermore, Figure~\ref{fig:id_curve} suggests that the CTH-RV and the FTL are practically unidentifiable no matter what output error threshold to chose, under the given initial and input conditions. I.e., there is no guarantee to find the unique parameters for the CTH-RV and the FTL even with perfect (error-free) data. On the other hand, the OV and the IDM can be practically identifiable when the threshold for output error is decreased. However, identifiability may not be achieved if the data is contaminated with noises ($\epsilon$ is large).

In summary, the \textit{direct test} is a conceptually straightforward method to test identifiability given a particular experimental design (i.e., with a known initial condition and a known input). \edit{It is very helpful to know (for model that is structurally locally identifiable) if the experiment leads to practical unidentifiability. Our approach provides a tool that highlights where an experiment can fail in the sense that similar outputs are obtained by distinct parameters}. This practical identifiability test is complementary to that of \textit{structural identifiability} analysis, which provides theoretical identifiability results for almost all initial conditions and parameters, but does not preclude the existence of special initial conditions and parameters (those belong to a set of measure zero) that result in non-identifiability (shown in Proposition~\ref{prop2} and \ref{prop3}). Note also that the \textit{direct test} does not aim to find \textit{all} initial conditions and parameters that fall into a set of measure zero, but simply tests if indistinguishable parameters exist given a full experimental setup, which can be useful in practice. Having a user-defined output difference threshold $\epsilon$ gives the \textit{direct test} more flexibility to understand the impacts of output measurement errors.

However, numerical test also has some disadvantages. For example, depending on the optimization solver and the set of starting points, the numerical approach may not find the true global maximum, and thus may misleadingly provide a lower bound on $\delta^*$, rather than the global maximum. In addition, the numerical direct test is performed based on a specific initial condition and input trajectory that are known, which limits its use to inform experimental designs. \edit{This approach does not consider practical experimental design, including how to design one or several experiments to best estimate the parameters, which is an interesting direction for further investigation.} Nevertheless, it is insightful as a verification tool to diagnose the potential unidentifiability under a specific setup.

\section{Conclusion}
In this work, we study the structural and practical identifiability of four car-following models. 

The structural identifiability analysis is carried through using a \textit{differential geometry} approach. It provides theoretical identifiability results under error-free assumptions for almost all initial conditions and parameters, and informs the admissible input condition that enables structural identifiability. However, it does not preclude the existence of initial conditions and parameters that lose identifiability. The results from the \textit{differential geometry} analysis show that all models are structurally locally identifiable, i.e., all models in theory have unique parameterizations when fitting with space gap data given almost all initial conditions and any admissible input. For some initial conditions such as an equilibrium initial condition, higher-order input is required for CTH-RV, OV and IDM to be identifiable. No admissible input exists for FTL when \newedit{the} initial condition is at equilibrium to enable identifiability.

As a complementary analysis, we use a numerical \textit{direct test} to study practical identifiability given a specific experimental setup (a given known initial condition and input). It also provides insights on the sensitivity of parameter identifiability due to measurement errors on the output. The \textit{direct test} finds that CTH-RV and FTL are practically unidentifiable, meaning that there exists multiple distinct parameters that produce the same output, given a specific experimental setup. It also suggests that OV and IDM are practically identifiable when the measurement error is small. Although only four models are investigated, the provided methods can be applied to other car following models.

The findings also open up questions for future research. For example, we are interested to design an experiment or sets of experiments that avoid problematic initial conditions for identifiability. Design of reduced order models that allow robust parameter identification while reflecting the dynamics of car following behavior are also of interest.  

\section{Acknowledgement}
    This material is based upon work supported by the National Science Foundation under Grant No. CMMI-1853913 and the USDOT Dwight D. Eisenhower Fellowship program under Grant No. 693JJ32145022. 
    It is also supported by the Inria Associated Team MEMENTO “Modeling autonomous vehicles in traffic flow.”
\appendix
\section{Appendix}
\subsection{An example for structurally locally identifiable CTH-RV}
\label{appendix:example_rank5}
Given the symbolic matrix~\eqref{eq:O_1}, one can find numerical substitutions for all the symbolic variables such that~\eqref{eq:O_1} is full rank. \newedit{Such substitutions belong to a generic initial condition and parameter set. One example is:}
\begin{equation}
    [k_1,k_2,\tau, u_0,v_0,s_0]=[0.01,0.12,1.4,30,33,40].
\end{equation}
The corresponding identifiability matrix becomes:
\begin{equation}
\label{eq:O_1_rank5}
\begin{aligned}
    &\mathcal{O}_I (\tilde{x}_0,u) = \\
    &\begin{bmatrix}
    1 & 0 & 0 & 0 & 0\\
    0 &-1 & 0 & 0 & 0\\
    -0.0100  & 0.134&  6.20& 3.00& 0.330\\
     0.00134 &-0.00796 & 1.579&-0.824&-0.0484\\
    -0.0000796 & -0.000274  &-0.658&  0.107 &  0.003456
    \end{bmatrix}
    \end{aligned},
\end{equation}
and rank($\mathcal{O}_I (\tilde{x}_0,u)$)=5.

\subsection{Proof for proposition~\ref{prop3}}
\label{appendix:proof_prop3}
Recall the system of equations for CTH-RV:

\begin{equation}
    \label{eq:CTH-FL}
    \begin{aligned}
    \dot{x}(t) = 
    \begin{bmatrix}
    \dot{s}(t)\\
    \dot{v}(t)
    \end{bmatrix}
    &=
    \begin{bmatrix}
    u(t)-v(t)\\
    k_1(s(t) - \tau v(t)) + k_2(u(t) -  v(t))  
    \end{bmatrix}\\
     y(t) &= x(t)
     \end{aligned}
\end{equation}
This system in fact can be written as a scalar differential equation:
\begin{equation}
\label{eq:v_d}
    \dot{v}(t) = k_1 \int_0^t{(u(\xi)-v(\xi))}d\xi+(-k_1\tau-k_2)v(t)+k_2u(t).
\end{equation}
Take the derivative on both sides to get rid of the integral operator:
\begin{equation}
\label{eq:v_dd}
    \ddot{v}(t) = k_1 {(u(t)-v(t))}dt+(-k_1\tau-k_2)\dot{v}(t)+k_2\dot{u}(t).
\end{equation}
Re-arrange to separate the input and output:
\begin{equation}
\label{eq:v_dd_rearrange}
    \ddot{v}(t)+(k_1\tau+k_2)\dot{v}(t)+k_1 v(t) = k_1u(t)+k_2\dot{u}(t).
\end{equation}
Now this is the standard form of a second-order non-homogeneous differential equation. The solution $v(t)$ has two parts, a homogeneous solution (or complementary solution $v_c(t)$, corresponds to the transient response), and a particular solution (or $V_p(t)$, corresponds to the steady state response). The general solution would be the sum of two parts:
\begin{equation}
    v(t) = v_c(t) + V_p(t).
\end{equation}
First let's see the homogeneous solution $v_c(t)$ by solving
\begin{equation}
\label{eq:vc}
    \ddot{v}(t)+(k_1\tau+k_2)\dot{v}(t)+k_1 v(t) = 0.
\end{equation}
The solution is of the form:
\begin{equation}
    v_c(t) = c_1 e^{r_1t}+c_2 e^{r_2t},
\end{equation}
where $r_1$ and $r_2$ are the roots of the characteristic equation for~\eqref{eq:vc}, and $c_1$ and $c_2$ are constants, which can be solved using initial values after the particular solution $V_p(t)$ is obtained.

Moving on to $V_p(t)$. We will show the condition such that the steady state response $V_p(t)$ does not depend on $k_1$.

For simplicity, let us assume the input (forcing) function as a sinusoidal wave, i.e., $u(t) = a\text{sin}(\omega t)$. Because we can represent almost all functions by Fourier series, which is a superposition of sine and cosine waves, the choice of $u(t)$ does not change the solution. The forcing term, or the RHS of~\eqref{eq:v_dd_rearrange} is
\begin{equation}
    g(t) = k_1u(t)+k_2\dot{u}(t) = k_1a\text{sin}(\omega t)+k_2a\omega \text{cos}(\omega t).
\end{equation}
By undetermined coefficients method, we can solve for the particular solution and the corresponding coefficients:
\begin{equation}
    V_p(t) = A\text{sin}(\omega t)+B\text{cos}(\omega t),
\end{equation}
where $A,B$ can be solved by plugging $V_p(t)$ into~\eqref{eq:v_dd_rearrange}. $A,B$ can be represented in terms of $a, \omega$ and the model parameters $\theta$. We will simplify the expressions by denoting $A,B$ as functions of some parameters:
\begin{equation}
    \begin{aligned}
    A &= f_A(a,\omega,\theta)\\
    B &= f_B(a,\omega,\theta).
    \end{aligned}
\end{equation}
Recall the superposition principle for ODE since ODE is a linear operator, the particular solution will always be a combination of sine and cosine waves with the same frequency but different magnitudes as the input $u(t)$. In this case, we just need to see if $A,B$ depend on any parameter $p \in \{k_1, k_2,\tau\}$ to make a difference in the steady state response. This can be done by taking the partial derivatives $\dfrac{\partial f_A}{\partial p}$ and $\dfrac{\partial f_B}{\partial p}$ and examine the dependency. In other words:
\begin{equation}
\label{eq:find}
\begin{aligned}
&\text{find the solution $\mathcal{S}_p$ for}\\
&\dfrac{\partial f_A}{\partial p}=0 \text{ and }\ \dfrac{\partial f_B}{\partial p}=0,\ \forall \ p\in \{k_1, k_2,\tau\} 
\in \mathcal{P}
\end{aligned}
\end{equation}
where $\mathcal{P}$ is the set of legal values for parameter $p$.
If the solution set for each $p$, $\mathcal{S}_p$, does not depend on $p$, then $p$ is unidentifiable. Solve for~\eqref{eq:find} we get
\begin{equation}
\label{eq:sol1}
\begin{aligned}
    \mathcal{S}_{k_1}&=\{k_2 = \dfrac{1}{\tau}\}\\
    \mathcal{S}_{k_2}&=\emptyset\\
    \mathcal{S}_{\tau}&=\emptyset.
\end{aligned}
\end{equation}
This suggests that $A,B$ are independent of $k_1$ if $k_2 = \dfrac{1}{\tau}$. 

Now let us go back to the initial condition and the transient response. In order to use the information $s(0) = s_0$, we'll need to work on the first-order differential equation~\eqref{eq:v_d} instead of~\eqref{eq:v_dd_rearrange}. Simply plug in the initial condition $s(0) = s_0$ and $v(0) = v_0$ in~\eqref{eq:v_d} to obtain:
\begin{equation}
\label{eq:IVP}
    \dot{v}(0) = k_1 s_0+(-k_1\tau-k_2)v_0+k_2u_0,
\end{equation}
where $\int_0^t{(u(\xi)-v(\xi))}d\xi |_{t=0} = s_0$, and $u_0$ is the initial value of the input. Now if we do the same as before~\eqref{eq:find} by taking partial derivatives of $\dot{v}(0)$ to find the unidentifiable parameter(s) and the corresponding ``condition" in order for that parameter to be unidentifiable, i.e.,
\begin{equation}
\label{eq:findp_vdot0}
\begin{aligned}
\text{for }p\in \{k_1, k_2,\tau\} &\text{, find the solution $\mathcal{S}_p$ for}\\
\dfrac{\partial \dot{v}(0)}{\partial p}=0 \ &\forall
\ p\in \mathcal{P}
\end{aligned}
\end{equation}
Solving~\eqref{eq:findp_vdot0} we get:
\begin{equation}
\label{eq:sol2}
\begin{aligned}
    \mathcal{S}_{k_1}&=\{s_0=\tau v_0\}\\
    \mathcal{S}_{k_2}&=\emptyset\\
    \mathcal{S}_{\tau}&=\emptyset.
\end{aligned}
\end{equation}
The solution shows that $k_1$ cannot be identified if $\tau=\dfrac{s_0}{v_0}$ from the transient response. Substitute $\tau=\dfrac{s_0}{v_0}$ and $k_2=1/\tau=\dfrac{v_0}{s_0}$ into the $\mathcal{O}_I (\tilde{x}_0,u)$, we obtain:
\begin{equation}
\label{eq:O_appendix}
\begin{aligned}
    &\mathcal{O}_I (\tilde{x}_0,u) = \\
    &\begin{bmatrix}
    1 & 0 & 0 & 0 & 0\\
    0 &-1 & 0 & 0 & 0\\
    -k_1 & v_0/s_0+k_1s_0/v_0 & 0 & v_0-u(t) & k_1 v_0\\
    -k_1(v_0/s_0+k_1s_0/v_0) & (v_0/s_0+k_1s_0/v_0)^2-k_1 & 0 & o_{44}& o_{45}&\\
    o_{51}& o_{52}& 0 &o_{54}& o_{55}&
    \end{bmatrix}
    \end{aligned}
\end{equation}

\begingroup
\allowdisplaybreaks
\begin{align*}
o_{44} &= -\dot{u}(t)-(v_0/s_0+k_1s_0/v_0)(v_0-u(t))-v_0(v_0-u(t))/s_0\\
o_{45} &= -k_1v_0(v_0/s_0+k_1s_0/v_0)-k_1v_0(v_0-u(t))/s_0\\
o_{51} & =k_1 (k_1 - (v_0/s_0 + k_1 s_0/v_0)^2)\\
o_{52} & =-(v_0/s_0+k_1s_0/v_0)(k_1-(v_0/s_0+k_1s_0/v_0)^2)-k_1(v_0/s_0+k_1s_0/v_0)\\
o_{54} & = \dot{u}(t)(2v_0/s_0+k_1 s_0/v_0)-k_1(v_0-u(t))-(k_1-(v_0/s_0+k_1 s_0/v_0)^2)\\&(v_0-u(t))-\ddot{u}(t)+(v_0(2v_0/s_0+2k_1 s_0/v_0))/s_0\\
o_{55} &= k_1v_0\dot{u}(t)/s_0-k_1v_0(k_1-(v_0/s_0+k_1 s_0/v_0)^2)-k_1^2(v_0-u(t))+\\&k_1v_0(v_0/s_0 + k_1 s_0/v_0)(v_0-u(t))^2/s_0.
\end{align*}
\endgroup
Matrix~\eqref{eq:O_appendix} clearly shows that the above substitution results in rank($\mathcal{O}_I (\tilde{x}_0,u)$)=4. The column corresponding to parameter $k_1$ is zero, meaning that $k_1$ is unidentifiable. This result is in agreement with the analytical proof that $k_1$ does not affect either the transient response or the steady state response. Note that we can choose $u_0 \neq v_0$ such that the initial condition of the system is at non-equilibrium. Please see Figure~\ref{fig:contour_CTHRV_3} for a visualization. Hence we complete the proof.


\bibliographystyle{abbrv}
\bibliography{refs}
\end{document}